\documentclass[10pt]{article}
\usepackage{amsfonts}
\usepackage{amsmath,amssymb}
\usepackage[dvips]{graphicx}
\usepackage{algorithmicx}
\usepackage{algpseudocode}
\usepackage{yhmath}
\usepackage{qtree}
\usepackage{xcolor}
\usepackage{epsfig}
\usepackage{lineno}
\usepackage{mathtools}
\usepackage{amsmath,bm}
\usepackage{amsthm}
\usepackage[T1]{fontenc}
\usepackage[utf8]{inputenc}
\usepackage{enumitem}
\usepackage{multirow}
\usepackage{makecell}
\usepackage{algorithm}

\usepackage{authblk}
\usepackage{bigints}

\usepackage[toc]{appendix}
\numberwithin{equation}{section}

\usepackage{subfigure}
\usepackage{verbatim}
\usepackage{graphicx}
\usepackage{geometry}
\usepackage{epsfig}
\usepackage{lineno,hyperref}
\usepackage[many]{tcolorbox}
\usetikzlibrary{shadows}
\usepackage{caption}
\usepackage{hyperref}

\newtcolorbox{shadedbox}{
  breakable,
  enhanced jigsaw,
  colback=white,
}

\newtheorem{assumption}{Assumption}
\newtheorem{theorem}{Theorem}[section]
\newtheorem{lemma}{Lemma}[section]
\newtheorem{remark}{Remark}[section]

\newcommand{\dx}{\,{\mathrm d}x}

\newcommand{\dt}{\,{\mathrm d}t}

\makeatletter
\def\BState{\State\hskip-\ALG@thistlm}
\makeatother

\newcommand{\lc}{\mathrel{\raise2pt\hbox{${\mathop<\limits_{\raise1pt\hbox
{\mbox{$\sim$}}}}$}}}
\newcommand{\gc}{\mathrel{\raise2pt\hbox{${\mathop>\limits_{\raise1pt\hbox
{\mbox{$\sim$}}}}$}}}
\newcommand{\ec}{\mathrel{\raise1pt\hbox{${\mathop=\limits_{\raise2pt\hbox
{\mbox{$\sim$}}}}$}}}

\renewcommand{\theequation}{\arabic{section}.\arabic{equation}}

\renewcommand{\thefigure}{\arabic{figure}}%

\usepackage{fancyhdr}
\begin{document}

\title{Stability and convergence of multi-product expansion splitting methods with negative weights for semilinear parabolic equations}
\author[1]{Xianglong Duan}
\author[2]{Chaoyu Quan}
\author[3]{Jiang Yang}
\author[4]{Zijing Zhu}

\date{}
\affil[1]{\small School of Mathematical Sciences, Capital Normal University, Beijing, 100048, People's Republic of China
(\href{mailto:xianglong.duan@cnu.edu.cn}{xianglong.duan@cnu.edu.cn}).}
\affil[2]{\small Division of Mathematics, School of Science and Engineering, The Chinese University of Hong Kong, Shenzhen, 518172, Guangdong, China
(\href{mailto:quanchaoyu@cuhk.edu.cn}{quanchaoyu@cuhk.edu.cn}).}
\affil[3]{\small  Department of Mathematics, SUSTech International Center for Mathematics \& National Center for Applied Mathematics Shenzhen (NCAMS), Southern University of Science and Technology, Shenzhen, People's Republic of China 
(\href{mailto:yangj7@sustech.edu.cn}{yangj7@sustech.edu.cn})}
\affil[4]{\small Department of Mathematics, Southern University of Science and Technology, Shenzhen, 518055, Guangdong, People's Republic of China
(\href{mailto:zhuzijing523@outlook.com}{zhuzijing523@outlook.com}).}

\maketitle

\begin{abstract}
The operator splitting method has been widely used to solve differential equations by splitting the equation into more manageable parts.  
In this work, we resolves a long-standing problem---how to establish the stability of multi-product expansion (MPE) splitting methods with negative weights.
The difficulty occurs because negative weights in high-order MPE method cause the sum of the absolute values of weights larger than one, making standard stability proofs fail.
In particular, we take the semilinear parabolic equation as a typical model and establish the stability of arbitrarily high-order MPE splitting methods with positive time steps but possibly negative weights. Rigorous convergence analysis is subsequently obtained from the stability result. Several numerical experiments validate the stability and accuracy of various high-order MPE splitting methods, highlighting their efficiency and robustness.
\end{abstract}



\section{Introduction} \label{section:intro} 
Operator splitting method plays a critical role in approximating solutions of differential equations involving the coupling of multiple physical fields. The core idea is to subdivide the original equations into a series of simpler problems, each corresponding to a different physical or mathematical process. By solving each simpler problem independently within a time step and subsequently recombining them, operator splitting method provides an approximate solution to the original problem (see reviews \cite{mclachlan2002splitting,blanes2024splitting}). This method is particularly advantageous for problems in which multiple processes, such as diffusion, advection, and reaction, interact simultaneously, as operator splitting method allows these processes to be separated and then treated with specialized methods optimized for each.

Consider a general evolutionary problem of the form  
\begin{equation}\label{eq:evolutionary problem}
    \left\{
    \begin{aligned}
       & u'(t)      = A\left(u(t)\right) + B\left(u(t)\right), \quad t\in [0,T],  \\
       & u(0)~\text{given},         
    \end{aligned}
    \right.
\end{equation}
involving (possible unbounded and nonlinear) operators $A: D(A) \subset X \rightarrow X$ and $B: D(B) \subset X \rightarrow X$; throughout, we tacitly assume that the domains $D(A)$ and $D(B)$ are suitably chosen subsets of the underlying Banach space $(X,\|\cdot\|_{X})$  with non-empty intersection $D(A) \cap D(B) \neq \emptyset$. 
The exact solution to \eqref{eq:evolutionary problem} is denoted by 
\begin{equation}
    \begin{aligned}
        u(t) = \mathcal{E}_{A+B}(t)u(0),
    \end{aligned}
\end{equation}
with the evolution operator $\mathcal{E}_{A+B}$. The evolutionary problem \eqref{eq:evolutionary problem} is subdivided into two problems $u'=A(u)$ and $u'=B(u)$, whose solutions correspond to evolution operators $\mathcal{E}_{A}$ and $\mathcal{E}_{B}$, respectively. 
By alternately employing these two evolution operators $\mathcal{E}_{A}$ and $\mathcal{E}_{B}$, a $k$th-order splitting operator $\mathcal{S}^{[k]}$ can be systematically generated. Let $t_n=n\tau$ for $n=0,1,\cdots,N$, where the step size $\tau = T/N$ with $N\geq1$. The corresponding splitting method is written as
\begin{equation}
    \begin{aligned}
        u^{n} = \mathcal{S}^{[k]}(\tau)u^{n-1} \approx u(t_n) = \mathcal{E}_{A+
    B}(\tau)u(t_{n-1}),\quad 1\leq n\leq N
    \end{aligned}
\end{equation}
that approximates the exact solution $u(t_n)$ to \eqref{eq:evolutionary problem}. 
Typically, the splitting operator $\mathcal{S}^{[k]}$ is composed by a single-product of the two evolution operators $\mathcal{E}_{A}$ and $\mathcal{E}_{B}$:
\begin{equation}\label{eq:SPE_scheme1}
    \begin{aligned}
        \mathcal{S}^{[k]}(\tau
        ) = \mathcal{E}_B(b_{m_i}\tau) \mathcal{E}_A(a_{m_i}\tau)\cdots\mathcal{E}_B(b_{1}\tau)\mathcal{E}_A(a_{1}\tau),
    \end{aligned}
\end{equation}
and multi-product expansion (MPE) splitting method is a linear combination of the splitting operators:
\begin{equation}\label{eq:MPE_scheme1}
    \begin{aligned}
        \mathcal{S}^{[k]}(\tau
        ) =\sum_{i=1}^{M}c_i \mathcal{E}_B(b_{i,m_i}\tau) \mathcal{E}_A(a_{i,m_i}\tau)\cdots\mathcal{E}_B(b_{i,1}\tau)\mathcal{E}_A(a_{i,1}\tau),
    \end{aligned}
\end{equation}
where the time step coefficients $a_{i,j},b_{i,j}$ can be real or complex, the weights $c_i$ are usually real and $M$, $m_i$ are some positive integers.

The earliest developments in operator splitting methods began with two splittings: the first-order Lie--Trotter splitting
\begin{equation}
    \begin{aligned}
        \mathcal{S}^{[1],1}(\tau) = \mathcal{E}_A(\tau)\mathcal{E}_B(\tau) \quad\mbox{or}\quad\mathcal{S}^{[1],2}(\tau)=\mathcal{E}_B(\tau)\mathcal{E}_A(\tau),
    \end{aligned}
\end{equation}
and the second-order Strang splitting
\begin{equation}\label{eq:Strang_scheme}
    \begin{aligned}
        \mathcal{S}^{[2],1}(\tau) = \mathcal{E}_A\left(\frac{\tau}{2}\right)\mathcal{E}_B(\tau)\mathcal{E}_A\left(\frac{\tau}{2}\right)\quad\mbox{or}\quad \mathcal{S}^{[2],2}(\tau)=\mathcal{E}_B\left(\frac{\tau}{2}\right)\mathcal{E}_A(\tau)\mathcal{E}_B\left(\frac{\tau}{2}\right).
    \end{aligned}
\end{equation}
Historically, a consensus has emerged regarding the development of splitting techniques based on the Lie product formula \cite{reed1980methods}. However, the precise origins of this formula remain somewhat ambiguous  \cite{chorin1978product,glowinski2017splitting,blanes2024splitting,mclachlan2002splitting}. Trotter later generalized it to self-adjoint linear operators in \cite{trotter1959product}, leading to the first-order Lie--Trotter splitting method. The second-order Strang splitting method is constructed by composing the Lie--Trotter splitting and its adjoint with half time steps. It was first introduced in \cite{strang1968construction} as an alternative approach to solving multidimensional problems with one-dimensional operators. As mentioned in
 \cite{blanes2024splitting}, the Strang splitting has been a popular integrator, as second-order methods often achieve the right balance between accuracy and complexity.
In some practical applications \cite{laskar1989numerical,sussman1992chaotic,laskar2004long}, higher accuracy might be required, whereas the Lie--Trotter and Strang splitting methods, while describing the systems well \cite{glowinski2017splitting}, are limited to first- and second-order accuracy.

One technique to derive high-order (third-order and higher) splitting methods \eqref{eq:SPE_scheme1} is the composition of lower-order splitting methods. For example, one can utilize the Strang splitting as a basic method and impose the symmetry of step coefficients on them to construct high-order splitting methods. In \cite{creutz1989higher}, it was shown that even-order splitting methods can be recursively constructed through compositions of the Strang splitting method. In \cite{yoshida1990construction}, a series of time-symmetric even-order splitting methods with exact time reversibility such as triple jump \cite{creutz1989higher} are derived recursively from the Strang splitting method by leveraging the property that even-order terms in the Baker--Campbell--Hausdorff (BCH) formula can be canceled.  Analogously, one can develop higher-order splitting methods by composing the Lie--Trotter splitting method with its adjoint while leveraging order conditions on step coefficients. These methods have been examined in \cite{cervi2018high}, which provides an overview of several third-order operator splitting methods. Beyond composition, step coefficients can also be determined using various order conditions summarized by Blanes et al. in \cite{blanes2024splitting} to construct higher-order splitting methods.  In addition, a broader range of high-order splitting methods is comprehensively summarized in \cite{blanes2024splitting,mclachlan2002splitting}.  A key advantage of splitting methods is that the resulting algorithms can preserve structure, such as symplecticity for Hamiltonian systems. Despite the ability of splitting methods to achieve arbitrarily high order and preserve structure for certain models, they face a fundamental challenge beyond the second order. As noted in \cite{castella2009splitting,hansen2009high}, high-order splitting methods cannot be applied to time-irreversible or semigroup problems due to the exponentially growing number of operators with unavoidable negative step coefficients \cite{sheng1989solving,suzuki1991general}.  Even for time-reversible systems where backward time integration is not problematic, the exponential growth on the number of force evaluations renders high-order symplectic integrators difficult to be derived and expensive to use, as demonstrated in \cite{chin2011multi}.

Complex step coefficients can be admitted in splitting methods \cite{suzuki1990fractal,bandrauk1991improved}. For parabolic equations, the existence of negative coefficients in high-order splitting methods leads to instability, while high-order splitting methods with complex step coefficients provide a possible way to overcome this order barrier \cite{castella2009splitting,hansen2009high}. To maintain the stability of the heat propagator, complex step coefficients with positive real parts are required. In \cite{hansen2009high}, the composition technique is applied to construct high-order splitting methods with complex step coefficients. However, as shown in \cite{blanes2013optimized}, if one considers complex step coefficients with positive real parts for splitting methods, there are still order barriers for the composition technique based on the Lie--Trotter and Strang splittings. Meanwhile, in \cite{bernier2023symmetric}, a symmetric-conjugate method applied to linear unitary
problems has been developed, which can achieve arbitrarily high order for several parabolic PDEs \cite{castella2009splitting,hansen2009high} because these methods of order higher than two do exist with coefficients having positive real parts.

The MPE splitting \eqref{eq:MPE_scheme1} generalizes the  typical splitting methods via linear combination, which enables the construction of arbitrarily high-order  splitting methods without complex or negative step coefficients. A second-order MPE splitting, first mentioned by Strang in \cite{strang1968construction} and named symmetrically weighted sequential splitting in \cite{csomos2005weighted}:
\begin{equation}\label{eq:sws_scheme}
    \begin{aligned}
        \mathcal{S}^{[2],3}\left(t\right) =  \frac{1}{2}\mathcal{E}_B(\tau)\mathcal{E}_A(\tau)+\frac{1}{2}\mathcal{E}_A(\tau)\mathcal{E}_B(\tau)
    \end{aligned}
\end{equation}
is derived from a linear combination of two Lie--Trotter splitting methods. In \cite{chin2010multi,chin2011multi}, an efficient approach is presented for constructing high-order integrators based on time-symmetric splitting methods (for example, Strang splitting) and Richardson extrapolation.  Generally, MPE splitting methods do not preserve the geometric
properties, since they are linear combinations of splitting methods. In certain cases, they can preserves the geometric properties (e.g., symplecticity for Hamiltonian problems) and achieve arbitrarily high order \cite{blanes1999extrapolation,chan2000extrapolation,blanes2017concise}. However, as rigorously shown by Suzuki and Sheng in \cite{sheng1989solving,suzuki1991general} and further strengthened in \cite{goldman1996n}, regardless of how many operators are subdivided, any MPE splitting method of order higher than two inevitably requires at least one negative step coefficient or weight per operator. Since negative step coefficients lead to instability of some operators, we consider the MPE splitting method with all positive step coefficients, which implies the inevitable existence of negative weights.   In other words, the MPE splitting methods we are concerned with are linear combinations of splitting methods up to order 2, with at least one negative weight.

Despite numerous numerical experiments and practical applications that demonstrate the remarkable performance of high-order MPE splitting methods with all positive step coefficients \cite{schatzman1994higher,chin2011multi,amiranashvili2021additive}, a rigorous analysis of their stability remains lacking. The primary challenge lies in the fact that trigonometric inequalities cannot be directly applied for estimation, as their upper bounds grow exponentially with iterations. In \cite{schatzman2002toward}, the author suggested that the safest approach to designing high-order operator splitting methods involves negative weight coefficients while ensuring the positivity of step coefficients. In the book \cite[Section 1.3]{hundsdorfer2013numerical}, Hundsdorfer and Verwer stated that ``{\em it is not precisely known for what kinds of problems this scheme (a MPE splitting with negative weights) will be stable or unstable}''. The stability analysis of MPE splitting methods with negative weights remains an open problem.

In this work, we establish the stability and consequently sharp convergence for
arbitrarily high-order MPE splitting methods. For simplicity, we consider the semilinear parabolic
equation with periodic boundary conditions as a typical problem. Our investigation addresses two main aspects. Firstly, to establish stability, we reformulate MPE splitting as several components, each formed by the composition of the three evolution operators $\mathcal{E}_A(\tau)$, $\mathcal{E}_B(\tau)$, and $\mathcal{E}_B(\tau)-\mathcal{I}$. By leveraging the stability properties of these evolution operators, we then derive Sobolev bounds for arbitrarily high-order MPE splitting methods. Secondly, we employ Lie derivative calculus and bounds for iterated Lie-commutators \cite{jahnke2000error, lubich2008splitting, thalhammer2012convergence} to analyze consistency. By combining stability and consistency, we rigorously prove convergence by deriving a Sobolev bound for the exact solution to the semilinear parabolic equation. In addition, we perform several numerical experiments to validate the accuracy, stability, and
adaptive time-stepping strategies of these high-order MPE splitting methods, highlighting their efficiency and
robustness through a range of nonlinear semilinear parabolic models.

The paper is organized as follows. In Section \ref{section:Notation and preliminaries}, we introduce some notations and preliminaries. 
Section \ref{section:MPE for rd equation} summarizes existing formulations of MPE splittings and proposes two new fourth-order splittings. Section \ref{section:stability} states the semilinear parabolic equation as a typical equation and provides for the first time a stability result for generic MPE splitting methods. In Section \ref{section:convergence}, the consistency is established through the derivation of the local error expansion for arbitrarily high-order MPE splitting methods, followed by the proof of convergence. In Section \ref{section:numerical experiments}, we conduct several numerical experiments that validate our theoretical results, demonstrating the stability and accuracy of the proposed high-order MPE splitting methods in various test cases. Finally, we conclude the paper in Section \ref{section:conclusions}.

\section{Notations and preliminaries}
\label{section:Notation and preliminaries}

Throughout this paper, we formally define the following notations.  We denote by \( \mathbb{N} = \{ m \in \mathbb{Z} : m \geq 0 \} \) the set of non-negative integers. The compact multi-index notation \( \mu = (\mu_1, \ldots, \mu_d) \in \mathbb{Z}^d \) and the vector notation \( x = (x_1, \ldots, x_d) \in \mathbb{R}^d \) are employed.  Moreover, for \( \alpha, \mu \in \mathbb{Z}^d \), relations such as \( \leq \) are defined componentwise; that is, \( \alpha \leq \mu \) iff \( \alpha_j \leq \mu_j \) for all \( 1 \leq j \leq d \). We set \( |\mu| = \mu_1 + \cdots + \mu_d \) and \( \mu! = \mu_1! \cdots \mu_d! \) for \( \mu \in \mathbb{N}^d \), and \( c\mu = (c \mu_1, \ldots, c \mu_d) \) for \( c \in \mathbb{Z} \) and \( \mu \in \mathbb{Z}^d \). For \( \mu \in \mathbb{N}^d \) and \( x \in \mathbb{R}^d \), we write \( x^\mu = x_1^{\mu_1} \cdots x_d^{\mu_d} \) and \( \partial_x^\mu = \partial_{x_1}^{\mu_1} \cdots \partial_{x_d}^{\mu_d} \) for brevity and further denote by \( \Delta = \partial_{x_1}^2 + \cdots + \partial_{x_d}^2 \) the \( d \)-dimensional Laplace operator. 
 The Sobolev space \( W^{m, p}(\Omega) \) consists of functions whose partial derivatives up to order \( m \geq 1 \) belong to \( L^p(\Omega) \). In particular, the Hilbert spaces with periodic boundary conditions are \( H^0_{\text{per}}(\Omega) = L^2(\Omega) \) and \( H^m_{\text{per}}(\Omega) = W^{m, 2}_{\text{per}}(\Omega) \) for \( m \geq 1 \), with the associated norm given by
 \begin{equation}
\begin{aligned}
    \|u\|_{H^k} = \|u\|_{H^k_{\text{per}}(\Omega)} = \sqrt{\sum_{|\alpha| \leq k} \|\partial_x^\alpha u\|_2^2},
\end{aligned}
 \end{equation}
for brevity.  We make the convention that $|\Omega|^\frac{1}{\infty}=1$ where $|\Omega|$ is the volume of the bounded domain $\Omega$.

Furthermore, for a family of (unbounded nonlinear) operators $\left(F_{\ell}\right)_{j \leq \ell \leq k}$ the product is defined downward (on a suitably chosen domain), that is, we set
\begin{equation}\label{eq:prod}
    \begin{aligned}
        \prod_{\ell=j}^{k} F_{\ell}=F_{k} \cdots F_{j},\quad  \text{if}\quad j \leq k,
    \end{aligned}
\end{equation}
and $\prod_{\ell=j}^{k} F_{\ell}=\mathcal{I}$ if $j>k$ with $\mathcal{I}$ the identity operator.

\section{Some MPE splitting methods}\label{section:MPE for rd equation}
In this section, we present two different constructions of MPE splitting methods for solving the evolutionary problem \eqref{eq:evolutionary problem}. In particular, we summarize some well-known high-order MPE splitting methods and also propose two fourth-order MPE splitting methods.

Using the formal calculus of Lie derivative introduced in \eqref{appendix:lie derivative}, a $k$th-order splitting operator can also be expressed as
\begin{equation}\label{eq:SPE_scheme}
    \begin{aligned}
        \mathcal{S}^{[k]}(t) = \prod_{j=1}^{m} \mathrm{e}^{a_{m+1-j} t D_{A}} \mathrm{e}^{b_{m+1-j} t D_{B}} \approx  \mathcal{E}_{A+B}(t)=\mathrm{e}^{t D_{A+B}}, \quad 0 \leq t \leq T,
    \end{aligned}
\end{equation}
 and a $k$th-order MPE splitting operator can be expressed as
\begin{equation}\label{eq:MPE_scheme}
    \begin{aligned}
        \mathcal{S}^{[k]}(t) = \sum_{i=1}^{M}c_i\prod_{j=1}^{m_i} \mathrm{e}^{a_{i,m_i+1-j} t D_{A}} \mathrm{e}^{b_{i,m_i+1-j} t D_{B}} \approx \mathcal{E}_{A+B}(t)=\mathrm{e}^{t D_{A+B}}, \quad 0 \leq t \leq T,
    \end{aligned}
\end{equation}
where the product is defined downward as \eqref{eq:prod} and $\mathcal{I}$ is the identity operator. Alternatively, the splitting operators \eqref{eq:SPE_scheme} and the MPE splitting operators \eqref{eq:MPE_scheme} can be written as \eqref{eq:SPE_scheme1} and \eqref{eq:MPE_scheme1}, respectively, in the form of the composition of evolution operators. It
should also be noted that, different from \eqref{eq:SPE_scheme1} and \eqref{eq:MPE_scheme1}, the evolution operators in \eqref{eq:SPE_scheme} and  \eqref{eq:MPE_scheme} acts in reversed order (from right to left).

For the evolutionary problem \eqref{eq:evolutionary problem}, we first define the defect operator $\mathcal{D}(t)$ as 
\begin{equation}\label{eq:defect}
    \begin{aligned}
        \mathcal{D}(t) = \mathcal{S}^{[k]}(t) - \mathcal{E}_{A+B}(t) = \sum_{i=1}^{M}c_i\prod_{j=1}^{m_i} \mathrm{e}^{a_{i,m_i+1-j} t D_{A}} \mathrm{e}^{b_{i,m_i+1-j} t D_{B}} -  \mathrm{e}^{t D_{A+B}},\quad  0<t\leq T,
    \end{aligned}
\end{equation}
which is also known as the local splitting error of the MPE splitting operator $\mathcal{S}^{[k]}(t)$ in \eqref{eq:MPE_scheme}. Here  $\mathcal{E}_{A+B}(t)$ is the exact solution operator in \eqref{eq:MPE_scheme}.
If the defect operator satisfies
\begin{equation}\label{eq:local_error_order}
    \begin{aligned}
        \mathcal{D}(t) = \mathcal{S}^{[k]}(t) - \mathcal{E}_{A+B}(t) = O(t^{k+1}),
    \end{aligned}
\end{equation}
the operator splitting method is said to be of order $k$.

Some interesting MPE splitting methods have been proposed and studied in  the literature. 
For example, two third-order splitting operators were proposed in \cite{burstein1970third}: 
\begin{equation}\label{eq:S3_1}
    \begin{aligned}
    \mathcal{S}^{[3],1}\left(t\right) = &\frac{2}{3}\mathrm{e}^{\frac{1}{2}tD_A}\mathrm{e}^{tD_B}\mathrm{e}^{\frac{1}{2}tD_A} + \frac{2}{3}\mathrm{e}^{\frac{1}{2}tD_B}\mathrm{e}^{tD_A}\mathrm{e}^{\frac{1}{2}tD_B} -\frac{1}{6}\mathrm{e}^{tD_B}\mathrm{e}^{tD_A} - \frac{1}{6}\mathrm{e}^{tD_A}\mathrm{e}^{tD_B},
    \end{aligned}
\end{equation}
and 
\begin{equation}\label{eq:S3_2}
    \begin{aligned}
    \mathcal{S}^{[3],2}\left(t\right) =  \frac{9}{8}\mathrm{e}^{\frac{1}{3}tD_B}\mathrm{e}^{\frac{2}{3}tD_A}\mathrm{e}^{\frac{2}{3}tD_B}\mathrm{e}^{\frac{1}{3}tD_A}- \frac{1}{8}\mathrm{e}^{tD_B}\mathrm{e}^{tD_A}.
    \end{aligned}
\end{equation}
Based on Lie--Trotter splitting methods, a fourth-order splitting operator \eqref{eq:S4_1} is mentioned in \cite{amiranashvili2021additive}, where various MPE splitting operators were further explored, tested, and compared: 
\begin{equation}\label{eq:S4_1}
    \begin{aligned}
    \mathcal{S}^{[4],1}\left(t\right) = &\frac{2}{3}\mathrm{e}^{\frac{1}{2}tD_B}\mathrm{e}^{\frac{1}{2}tD_A}\mathrm{e}^{\frac{1}{2}tD_B}\mathrm{e}^{\frac{1}{2}tD_A} + \frac{2}{3}\mathrm{e}^{\frac{1}{2}tD_A}\mathrm{e}^{\frac{1}{2}tD_B}\mathrm{e}^{\frac{1}{2}tD_A}\mathrm{e}^{\frac{1}{2}tD_B} -\frac{1}{6}\mathrm{e}^{tD_A}\mathrm{e}^{tD_B} - \frac{1}{6}\mathrm{e}^{tD_B}\mathrm{e}^{tD_A}.
    \end{aligned}
\end{equation}

Using the order conditions, we can also construct the following two new fourth-order splitting operators.
\begin{equation}\label{eq:S4_2}
    \begin{aligned}
        \mathcal{S}^{[4],2}\left(t\right) = 
        &  \frac{2}{3}\mathrm{e}^{\frac{1}{4}tD_A}\mathrm{e}^{\frac{1}{2}tD_B}\mathrm{e}^{\frac{1}{2}tD_A}\mathrm{e}^{\frac{1}{2}tD_B}\mathrm{e}^{\frac{1}{4}tD_A}   + \frac{2}{3}\mathrm{e}^{\frac{1}{4}tD_B}\mathrm{e}^{\frac{1}{2}tD_A}\mathrm{e}^{\frac{1}{2}tD_B}\mathrm{e}^{tD_A}\mathrm{e}^{\frac{1}{4}tD_B} \\
        &-\frac{1}{6}\mathrm{e}^{\frac{1}{2}tD_A}\mathrm{e}^{tD_B}\mathrm{e}^{\frac{1}{2}tD_A} - \frac{1}{6}\mathrm{e}^{\frac{1}{2}tD_B}\mathrm{e}^{tD_A}\mathrm{e}^{\frac{1}{2}tD_B}, 
    \end{aligned}
\end{equation}
 and
\begin{equation}\label{eq:S4_3}
    \begin{aligned}
\mathcal{S}^{[4],3}\left(t\right) = 
        & \frac{4}{3}\mathrm{e}^{\frac{1}{8}tD_A}\mathrm{e}^{\frac{1}{4}tD_B}\mathrm{e}^{\frac{3}{8}tD_A}\mathrm{e}^{\frac{1}{2}tD_B}\mathrm{e}^{\frac{3}{8}tD_A}\mathrm{e}^{\frac{1}{4}tD_B}\mathrm{e}^{\frac{1}{8}tD_A} + \frac{4}{3}\mathrm{e}^{\frac{1}{8}tD_B}\mathrm{e}^{\frac{1}{4}tD_A}\mathrm{e}^{\frac{3}{8}tD_B}\mathrm{e}^{\frac{1}{2}tD_A}\mathrm{e}^{\frac{3}{8}tD_B}\mathrm{e}^{\frac{1}{4}tD_A}\mathrm{e}^{\frac{1}{8}tD_B}\\      
        & - \frac{5}{6}\mathrm{e}^{\frac{1}{4}tD_A}\mathrm{e}^{\frac{1}{2}tD_B}\mathrm{e}^{\frac{1}{2}tD_A}\mathrm{e}^{\frac{1}{2}tD_B}\mathrm{e}^{\frac{1}{4}tD_A}   - \frac{5}{6}\mathrm{e}^{\frac{1}{4}tD_B}\mathrm{e}^{\frac{1}{2}tD_A}\mathrm{e}^{\frac{1}{2}tD_B}\mathrm{e}^{tD_A}\mathrm{e}^{\frac{1}{4}tD_B}.\\
    \end{aligned}
\end{equation}

Richardson extrapolation is another efficient and convenient approach to construct high-order MPE splitting methods. Using the time reversibility of Strang splitting \( \mathcal{S}_2(t) \) in \eqref{eq:Strang_scheme}, for a given set of $k$ distinct numbers $\{\gamma_i\}_{i=1}^k$, the  $2k$-th order MPE splitting can be constructed \cite{chin2011multi} as follows
\begin{equation}\label{eq:S2k}
    \begin{aligned}
        \mathcal{S}^{[2 k]}\left(t\right): = \sum^{k}_{i=1}c_i \left[\mathcal{S}_2\left(\frac{t}{\gamma_i}\right)\right]^{\gamma_i},\quad \text{with}\quad c_i = \prod^k_{\substack{j=1 \\
 j\neq i}}\frac{\gamma_i^2 }{\gamma_i^2-\gamma_j^2},~k\geq 1.
    \end{aligned}
\end{equation}
The natural sequence $\{\gamma_i: \gamma_i=i\}_{i=1}^k$ produces $2 k$-th order splittings (for example, $k = 2,3,4,5$):
\begin{equation}\label{eq:richardson schemes}
    \begin{aligned}
\mathcal{S}^{[4],4}(t)&:=-\frac{1}{3} \mathcal{S}_{2}(t)+\frac{4}{3} \left[\mathcal{S}_{2}\left(\frac{t}{2}\right)\right]^2,\\
\mathcal{S}^{[6]}(t)&:=\frac{1}{24} \mathcal{S}_{2}(t)-\frac{16}{15} \left[\mathcal{S}_{2}\left(\frac{t}{2}\right)\right]^{2}+\frac{81}{40} \left[\mathcal{S}_{2}\left(\frac{t}{3}\right)\right]^{3},\\
\mathcal{S}^{[8]}(t)&:=-\frac{1}{360} \mathcal{S}_{2}(t)+\frac{16}{45} \left[\mathcal{S}_{2}\left(\frac{t}{2}\right)\right]^{2}-\frac{729}{280} \left[\mathcal{S}_{2}\left(\frac{t}{3}\right)\right]^{3}+\frac{1024}{315} \left[\mathcal{S}_{2}\left(\frac{t}{4}\right)\right]^{4}, \\
\mathcal{S}^{[10]}(t)&:=\frac{1}{8640} \mathcal{S}_{2}(t)-\frac{64}{945} \left[\mathcal{S}_{2}\left(\frac{t}{2}\right)\right]^{2}+\frac{6561}{4480} \left[\mathcal{S}_{2}\left(\frac{t}{3}\right)\right]^{3}
\\
&\quad-\frac{16384}{2835} \left[\mathcal{S}_{2}\left(\frac{t}{4}\right)\right]^{4}+\frac{390625}{72576} \left[\mathcal{S}_{2}\left(\frac{t}{5}\right)\right]^{5}.    
    \end{aligned}
\end{equation}
Note that the above mentioned MPE splittings have positive time steps but possibly negative weights. 

\begin{remark}
 Some other efficient MPE splitting methods  with negative time steps have recently been proposed in \cite{blanes2021novel,blanes2024generalized}. However, we will see in the following content that for these MPE splittings, the stability analysis in this work seems not applicable.
\end{remark}

\section{Stability of MPE splitting methods for semilinear parabolic equations}\label{section:stability}
In this section, we derive one of the main results of this work, the stability
 estimate for semilinear parabolic equations. We introduce the semilinear parabolic equations with periodic boundary conditions.  The analysis begins with several lemmas and theorems establishing the stability of the two evolution operators, followed by a complete stability analysis of arbitrarily high-order MPE splitting operators \eqref{eq:MPE_scheme} which is a highlight of this work. In the following study, we focus on high-order MPE splitting methods with positive time steps but possibly negative weights.

\subsection{Semilinear parabolic equation}
\label{subsection:semilinear equation}
Now, we consider the semilinear parabolic equation with a periodic boundary condition of the form: 
\begin{equation}\label{eq:semilinear equation}
    \left\{
    \begin{aligned}
       & u_{t}      = \Delta u+f(u), &  & x \in \Omega, ~  t \in(0,T],\\
       & u(x,0)  =u^0({x}),   &  & x \in \Omega,
    \end{aligned}
    \right.
\end{equation}
where the bounded spatial domain $\Omega =[-\pi,\pi]^d$ for $d=1,2,3$. The function $f:\mathbb{R}\rightarrow\mathbb{R}$ is nonlinear and $u^0(x)$ is the initial data. To establish the stability of $\mathcal{E}_B(\tau)$, we impose the following assumption on the function $f$:
\begin{assumption}\label{assumption:f'_bounded}
    There exists a positive constant $\kappa$ such that 
\begin{equation}\label{eq:f'_bounded}
    \underset{\xi\in\mathbb{R}}{\sup}|f'(\xi)|\leq \kappa,
\end{equation} 
where $f^\prime$ is the derivative of $f$.
\end{assumption}
\begin{remark}
From the existence and uniqueness theorem for nonlinear ODEs, if \( f \) is only locally Lipschitz, the solutions may not be defined for certain values of \( t \), even if \( f \) is smooth. For instance, the differential equation \( u_t= u^2 \) with initial condition \( u(0) = 1 \) has the solution \( u(t) = 1/(1 - t) \), which is not defined at \( t = 1 \). 
This implies that operator splitting methods can never be stable unconditionally. Therefore, Assumption \ref{assumption:f'_bounded} is reasonable. 
\end{remark}

\begin{remark}
If \eqref{assumption:f'_bounded} is not satisfied, one technique is to truncate $f'$ when it becomes excessively large so that this assumption can be guaranteed (cf. \cite[Eq.~(1.8)]{shen2010numerical} for the truncation technique), which will be further discussed in Section \ref{section:numerical experiments}.
\end{remark}


The initial-boundary value problem is interpreted as an evolutionary problem of the form \eqref{eq:evolutionary problem} by setting the unbounded linear operator \( A(u) = \Delta u \) and the (possible unbounded and nonlinear) operator \( B(u) = f(u) \) for \( u: \Omega \to \mathbb{R} \), with the Banach space \( X = L^p(\Omega) \) for $1\leq p\leq\infty$. Formally, the solution to the semilinear parabolic equation \eqref{eq:semilinear equation} is given by \( u(t) = \mathrm{e}^{t D_{A+B}} u^0 = \mathcal{E}(t) u^0 \). The evolution operator \( \mathcal{E}_A(t) = \mathrm{e}^{t \Delta} \) corresponds to the linear operator \( A \), which generates a semigroup \( (\mathrm{e}^{tA})_{t \geq 0} \) for \( u_t = \Delta u \), while the evolution operator \( \mathcal{E}_B(t) \) corresponds to \( u_t = f( u) \).

For splitting methods with real step coefficients, we recall that the evolution operator \( \mathcal{E}_A(t) \) is  unstable if $t$ is negative. So, it is natural to assume that the step coefficients $a_{i,j}$ in \eqref{eq:MPE_scheme} are positive. 
However, it has been proved theoretically in \cite{sheng1989solving,suzuki1991general,goldman1996n} that negative weights will exist in this case, that is, if \( k > 2 \) and all \( a_{i,j} > 0 \), then \( \min_i c_i < 0 \). Then, we study the MPE splitting operator defined in \eqref{eq:MPE_scheme} with all \( a_{i,j}, b_{i,j} \geq 0 \) and finite \( m_i \). We will later use MPE to refer to MPE splitting methods with all positive step coefficients.

%
%

\subsection{Stability of $\mathcal{E}_A(\tau)$ and $\mathcal{E}_B(\tau)$}\label{subsection: stability for suboperators}

Before analyzing the stability of high-order MPE splitting operators, we first present several lemmas to elucidate the properties of the two evolution operators $\mathcal{E}_A(\tau)$ and $\mathcal{E}_B(\tau)$.

\begin{lemma}[Stability of $\mathcal{E}_A(\tau)$]\label{lemma:E_A_stability}For any $\tau\geq 0 $ and bounded domain $\Omega$ with periodic boundary condition, the following estimates hold:
\begin{enumerate}
    \item[(1)] If $u^0 \in H^k_{\mathrm{per}}(\Omega)$ with $k > 0$, then
    $
    \|\mathcal{E}_A(\tau)u^0\|_{H^k} \leq \|u^0\|_{H^k}.
    $
    
    \item[(2)] If $u^0 \in L^p(\Omega)$ for $1 \leq p < \infty$, then
    $
    \|\mathcal{E}_A(\tau)u^0\|_{p} \leq \|u^0\|_{p}.
    $
    
    \item[(3)] If $u^0 \in L^\infty(\Omega)$, then
    $
    \|\mathcal{E}_A(\tau)u^0\|_{\infty} \leq \|u^0\|_{\infty}.
    $
    \end{enumerate}
\end{lemma}
\begin{proof}
(1) Consider the evolution operator   $\mathcal{E}_A(\tau)$ and set $u(t) =  \mathcal{E}_A(t) u^0$. For any multi-index $\alpha$ satisfying $0\leq|\alpha|\leq k$, we compute   
\begin{equation}
    \begin{aligned}
        \partial_t  \left\|\partial_x^\alpha u\right\|_2^2 &=2\int_{\Omega}  \partial_x^\alpha u \partial_x^\alpha u_t \dx  =2 \int_{\Omega} \partial_x^\alpha u  (\Delta \partial_x^\alpha u) \dx 
    =-2\left\|\nabla \partial_x^\alpha u\right\|_{2}^{2} \leq 0.
    \end{aligned}
\end{equation}
This implies that  $\left\|\partial_x^\alpha u\right\|_2\leq \left\|\partial_x^\alpha u^0\right\|_2$, which leads to the desired result.

(2) Let $u(t)=\mathcal{E}_A(t)u^0$ for $t\geq 0$.  
Consider a sequence of smooth convex functions \( g_\varepsilon(u) = (u^2+ \varepsilon^2)^{p/2} \) with small parameter $\varepsilon\in(0,1]$. Straightforward calculations give
\begin{equation}\label{eq:g''}
    \begin{aligned}
        \frac{\rm d}{\dt} \int_\Omega g_\varepsilon\left(u\right)\dx =  \int_\Omega g_\varepsilon'\left(u\right)\Delta u\dx = -\int_\Omega g_\varepsilon''\left(u\right)\left|\nabla u\right|^2\dx \leq 0, \quad \forall t\geq 0.
    \end{aligned}
\end{equation}
Then, for any $\varepsilon\in(0,1]$, we have
\begin{equation}\label{eq:g_varepsilon}
   \int_{\Omega}|u(t)|^p\dx\leq \int_{\Omega} g_\varepsilon(u(t))\dx \leq \int_{\Omega} g_\varepsilon(u^0)\dx.
\end{equation}
Since $\lim_{\varepsilon\rightarrow 0}g_{\varepsilon}(u^0)= |u^0|^p$ and $g_{\varepsilon}(u^0)\leq ((u^0)^2+1)^{p/2}$, we can directly obtain 
\begin{equation}
    \lim_{\varepsilon\rightarrow 0}\int_{\Omega} g_\varepsilon(u^0)\dx = \int_{\Omega}|u^0|^p\dx
\end{equation}
by applying the dominated convergence theorem. As a consequence, taking $\varepsilon\rightarrow0$ in \eqref{eq:g_varepsilon}, we then have the desired result.
    
(3) The linear Laplace operator $\Delta$ is the generator of a contraction semigroup $\{ \mathrm{e}^{t\Delta}\}_{ t\geq 0}$ and thus naturally possesses the contraction property that $\|\mathcal{E}_A(\tau)u^0\|_\infty\leq \|u^0\|_\infty$ (see, e.g., \cite[Lemma 2.1]{du2021maximum}). This completes the proof.
\end{proof}

In fact, the $H^k$-norm contractivity in Lemma \ref{lemma:E_A_stability} will be used in the subsequent consistency analysis.
The stability analysis of the nonlinear evolution operator $\mathcal{E}_B(\tau)$ is lengthier than that of the linear evolution operator $\mathcal{E}_A(\tau)$. To facilitate the stability and convergence analysis of the high-order MPE splitting operator, we present the following four lemmas concerning $\mathcal{E}_B(\tau)$ and $\mathcal{E}_B( \tau)-\mathcal{I}$, where $\mathcal{I}$ denotes the identity operator.

\begin{lemma}[Stability of $\mathcal{E}_B(\tau)-\mathcal{I}$]\label{lemma:E_B-I_stability}Suppose $w^0\in L^p(\Omega)$ for $1\leq p \leq \infty$, and \eqref{assumption:f'_bounded} is satisfied. Then, the following property holds:
\begin{equation}\label{eq:E_B-I}
    \left\|[\mathcal{E}_B(\tau)-\mathcal{I}] w^0\right\|_p \leq ({\mathrm{e}^{\kappa\tau}-1})\left(\left\|w^0\right\|_p+\frac{|\Omega|^{\frac{1}{p}}|f(0)|}{\kappa}\right),
\end{equation}
where $\mathcal{I}$ denotes the identity operator, $|\Omega|$ is the volume of the bounded domain $\Omega$, and $|\Omega|^\frac{1}{\infty}=1$ as aforementioned.
\end{lemma}
\begin{proof}
Suppose $w = \mathcal{E}_B(t)w^0$. By definition,
\begin{equation}
    \begin{aligned}
        \left\|[\mathcal{E}_B(\tau)-\mathcal{I}] w^0\right\|_p &=\left \|\int^\tau_0 \partial_t w \dt\right\|_p=\left \|\int^\tau_0 f(w)\dt\right\|_p \leq  \int^\tau_0 \left \|f(w)-f(w^0)\right\|_p \dt +\int^\tau_0 \left \|f(w^0)\right\|_p \dt\\
        &\leq  \kappa\int^\tau_0 \left \|w-w^0\right\|_p \dt +\tau\left \|f(w^0)\right\|_p =  \kappa\int^\tau_0 \left \|[\mathcal{E}_B(t)-\mathcal{I}] w^0\right\|_p \dt +\tau\left \|f(w^0)\right\|_p,\\
    \end{aligned}
\end{equation}
where $\kappa$ is introduced in \eqref{eq:f'_bounded}.
Applying Gr{\"o}nwall's lemma, we obtain
\begin{equation}
    \begin{aligned}
\left\|[\mathcal{E}_B(\tau)-\mathcal{I}] w^0\right\|_p &\leq \frac{\mathrm{e}^{\kappa\tau}-1}{\kappa}\left \|f(w^0)\right\|_p\leq \frac{\mathrm{e}^{\kappa\tau}-1}{\kappa}\left(\left \|f(0)\right\|_p+\kappa \left \|w^0\right\|_p\right)\\
&\leq({\mathrm{e}^{\kappa\tau}-1})\left(\left\|w^0\right\|_p+\frac{|\Omega|^{\frac{1}{p}}|f(0)|}{\kappa}\right).
    \end{aligned}
\end{equation}
This yields the desired result.
\end{proof}
\begin{lemma}[Stability of $\mathcal{E}_B(\tau)$]\label{lemma:E_B_stability}Suppose that $w^0\in L^p(\Omega)$ for $1\leq p\leq \infty$, and \eqref{assumption:f'_bounded} is satisfied. Then,  the following property holds:
    \begin{equation}\label{eq:E_B_stability}
        \left\|\mathcal{E}_B(\tau)w^0\right\|_p \leq \mathrm{e}^{\kappa\tau}(\left\|w^0\right\|_p +|\Omega|^{\frac{1}{p}}|f(0)|\tau).
    \end{equation}
\end{lemma}
\begin{proof}
    Using the triangle inequality, the result follows directly from \eqref{lemma:E_B-I_stability}
    \begin{equation}
    \begin{aligned}
 \left\|\mathcal{E}_B(\tau)w^0\right\|_p \leq   \left\|[\mathcal{E}_B(\tau)-\mathcal{I}] w^0\right\|_p +  \left\|w^0\right\|_p \leq \mathrm{e}^{\kappa\tau}(\|w^0\|_p +|\Omega|^{\frac{1}{p}}|f(0)|\tau),
    \end{aligned}
    \end{equation}
  where the final inequality is justified by the fact that $\mathrm{e}^{\kappa\tau}\leq 1+\kappa\tau \mathrm{e}^{\kappa\tau}$.
\end{proof}
\begin{lemma}[Stability of $\mathcal{E}_B(\tau)$]\label{lemma:E_B_associate}Suppose that $w^0,v^0\in L^p(\Omega)$, and \eqref{assumption:f'_bounded} is satisfied. Then, for almost every $x \in \Omega$ and all $\tau \geq 0$,
\begin{equation}
    \left|\mathcal{E}_B(\tau)w^0(x)-\mathcal{E}_B(\tau)v^0(x)\right|\leq \mathrm{e}^{\kappa\tau}\left|w^0(x)-v^0(x)\right|,
\end{equation}
which gives
\begin{equation}
    \begin{aligned}
        \left\|\mathcal{E}_B(\tau)w^0-\mathcal{E}_B(\tau)v^0\right\|_p\leq \mathrm{e}^{\kappa\tau}\left\|w^0-v^0\right\|_p.
    \end{aligned}
\end{equation}
\end{lemma} 

\begin{proof}
Suppose $w=\mathcal{E}_B(t)w^0$ and $v=\mathcal{E}_B(t)v^0$. We obtain that for almost any $x\in \Omega$, 
\begin{equation}\label{eq:pf_E_B_associate}
   \begin{aligned}
    \left|w(x,\tau) - v(x,\tau)\right|&=\left|\mathcal{E}_B(\tau)w^0(x)-\mathcal{E}_B(\tau)v^0(x)\right|=\left|\int^\tau_0 (\partial_t w-\partial_t v) \dt + w^0(x)-v^0(x)\right|\\
    &\leq \int^\tau_0 \left|f(w)-f(v)\right| \dt + \left|w^0(x)-v^0(x)\right|\leq \kappa\int^\tau_0 |w-v| \dt  + \left|w^0(x)-v^0(x)\right|.
\end{aligned} 
\end{equation}
By Gr{\"o}nwall's lemma, we have
\begin{equation}\label{eq:pf_E_B_associate1}
   \begin{aligned}
    |w(x,\tau) - v(x,\tau)|\leq \mathrm{e}^{\kappa\tau}\left|w^0(x)-v^0(x)\right|,
\end{aligned} 
\end{equation}
which also holds for the $L^p$-norm, where $1\leq p \leq \infty$.
\end{proof}

\begin{lemma}[Stability of $\mathcal{E}_B(\tau)-\mathcal{I}$]\label{lemma:E_B-I_associate}Suppose that $w^0,v^0\in L^p(\Omega)$ for $1\leq p\leq \infty$, and \eqref{assumption:f'_bounded} is satisfied. Then, for almost every $x \in \Omega$ and all $\tau \geq 0$,
    \begin{equation}
        \left|[\mathcal{E}_B(\tau)-\mathcal{I}]w^0
      (x)-[\mathcal{E}_B(\tau)-\mathcal{I}]v^0 (x)\right|\leq  \kappa \mathrm{e}^{\kappa\tau}\tau\left|w^0 (x)-v^0 (x)\right|,
    \end{equation}
which gives
\begin{equation}
    \left\|[\mathcal{E}_B(\tau)-\mathcal{I}]w^0-[\mathcal{E}_B(\tau)-\mathcal{I}]v^0\right\|_p\leq \kappa \mathrm{e}^{\kappa\tau}\tau\left\|w^0-v^0\right\|_p,
\end{equation}
    where  $\mathcal{I}$ is the identity operator.
\end{lemma}
\begin{proof}
Suppose $w=\mathcal{E}_B(t)w^0$ and $v=\mathcal{E}_B(t)v^0$. In a similar process as in \eqref{eq:pf_E_B_associate},  it follows that for almost any $x\in\Omega$, 
\begin{equation}
    \begin{aligned}
 \left|[\mathcal{E}_B(\tau)-\mathcal{I}]w^0 (x)-[\mathcal{E}_B(\tau)-\mathcal{I}]v^0 (x)\right|&=\left|\int_{0}^{\tau}\left(f(w) - f(v)\right) d t\right|\\
 &\leq  \kappa\int^\tau_0 \left|w-v\right| \dt \leq  \kappa \mathrm{e}^{\kappa\tau}\tau\left|w^0(x)-v^0(x)\right|,
\end{aligned}
\end{equation}
where the last inequality is derived from \eqref{lemma:E_B_associate}. This also holds for the $L^p$-norm, where $1\leq p \leq \infty$.
\end{proof}

\subsection{Stability of general MPE splitting methods}\label{subsection: stability for MPE}

With the properties of the evolution operator  $\mathcal{E}_A(\tau)$ and $\mathcal{E}_B(\tau)$ established above, we can now proceed with the stability analysis of MPE splitting operators. The following lemma lays the foundation for deriving the stability theorem.

\begin{lemma}\label{lemma:S-E_A} Suppose that $w^0,v^0\in L^p(\Omega)$, and \eqref{assumption:f'_bounded} is satisfied. Let the operator $\mathcal{S}(\tau) = \mathcal{E}_B\left(b_m\tau\right)\mathcal{E}_A\left(a_m\tau\right)\cdots \mathcal{E}_B\left(b_1\tau\right)\mathcal{E}_A\left(a_{1}\tau\right)$ with all $a_j,b_j\geq0$. Then, for $1\leq p\leq \infty$, it holds that
    \begin{equation}
    \begin{aligned}
         &\left\| [\mathcal{S}(\tau)-\mathcal{E}_A(\widetilde{a}\tau)]w^0-[\mathcal{S}(\tau)-\mathcal{E}_A(\widetilde{a}\tau)]v^0\right\|_p \leq\kappa \widetilde{b} \mathrm{e}^{\kappa \widetilde{b}\tau}\tau  \left\|w^0-v^0\right\|_p, \\
         &\left\|\mathcal{S}(\tau)w^0-\mathcal{E}_A(\widetilde{a}\tau)w^0\right\|_p \leq\left(\mathrm{e}^{\kappa\widetilde{b}\tau}-1\right)  \left(\left\|w^0\right\|_p+\frac{|\Omega|^{\frac{1}{p}}|f(0)|}{\kappa}\right),
         \end{aligned}
    \end{equation}
    where $\widetilde{a}=\sum_{j=1}^ma_j$ and  $\widetilde{b}=\sum_{j=1}^mb_j$.
\end{lemma}
\begin{proof}
    Firstly, define $\widetilde{\mathcal{S}}(\tau) = \mathcal{E}_B\left(b_{m-1}\tau\right)\mathcal{E}_A\left(a_{m-1}\tau\right)\cdots \mathcal{E}_B\left(b_1\tau\right)\mathcal{E}_A\left(a_{1}\tau\right)$,  $\widetilde{a}_s=\sum_{j=1}^{s}a_j$, and $\widetilde{b}_s=\sum_{j=1}^{s}b_j$, for $s = 1,2,\cdots,m$. Then, 
    \begin{equation}\label{eq:S-E_A}
       \mathcal{S}(\tau)-\mathcal{E}_A(\widetilde{a}_{m}\tau)=[\mathcal{E}_B(b_m\tau)-\mathcal{I}]\mathcal{E}_A(a_m\tau)\widetilde{\mathcal{S}}(\tau)+\mathcal{E}_A(a_m\tau)[\widetilde{\mathcal{S}}(\tau)-\mathcal{E}_A(\widetilde{a}_{m-1}\tau)].
    \end{equation}
    This yields that
    \begin{equation}\label{eq:L_1+L_2}
        \left\| [\mathcal{S}(\tau)-\mathcal{E}_A(\widetilde{a}_{m}\tau)]w^0-[\mathcal{S}(\tau)-\mathcal{E}_A(\widetilde{a}_{m}\tau)]v^0\right\|_p\leq {L}_1+{L}_2,
    \end{equation}
    where 
    \begin{equation}
        \begin{aligned}
          & {L}_1 = \left\|[\mathcal{E}_B(b_m\tau)-\mathcal{I}]\mathcal{E}_A(a_m\tau)\widetilde{\mathcal{S}}(\tau)w^0-[\mathcal{E}_B(b_m\tau)-\mathcal{I}]\mathcal{E}_A(a_m\tau)\widetilde{\mathcal{S}}(\tau)v^0\right\|_p, \\
    &  {L}_2 =\left\|\mathcal{E}_A(a_m\tau)[\widetilde{\mathcal{S}}(\tau)-\mathcal{E}_A(\widetilde{a}_{m-1}\tau)]w^0-\mathcal{E}_A(a_m\tau)[\widetilde{\mathcal{S}}(\tau)-\mathcal{E}_A(\widetilde{a}_{m-1}\tau)]v^0 \right\|_p.
        \end{aligned}
    \end{equation}
    By using \eqref{lemma:E_A_stability}, \eqref{lemma:E_B_associate}, and \eqref{lemma:E_B-I_associate},  we obtain
    \begin{equation}\label{eq:L_1}
    \begin{aligned}
         {L}_1
         &\leq \kappa b_m \mathrm{e}^{\kappa b_m\tau}\tau\left\|\mathcal{E}_A(a_m\tau)\left[\widetilde{\mathcal{S}}(\tau)w^0-\widetilde{\mathcal{S}}(\tau)v^0\right]\right\|_p\leq \kappa b_m\mathrm{e}^{\kappa b_m\tau}\tau\left\|\widetilde{\mathcal{S}}(\tau)w^0-\widetilde{\mathcal{S}}(\tau)v^0\right
        \|_p\\
         &\leq \kappa b_m\mathrm{e}^{\kappa b_m\tau}\tau  \mathrm{e}^{\kappa \widetilde{b}_{m-1}\tau} \left\|w^0-v^0\right\|_p = \kappa b_m  \mathrm{e}^{\kappa \widetilde{b}_m\tau}\tau \left\|w^0-v^0\right\|_p,
    \end{aligned}
    \end{equation}
and 
\begin{equation}\label{eq:L_2}
    \begin{aligned}
         {L}_2&\leq \left\|[\widetilde{\mathcal{S}}(\tau)-\mathcal{E}_A(\widetilde{a}_{m-1}\tau)]w^0-[\widetilde{\mathcal{S}}(\tau)-\mathcal{E}_A(\widetilde{a}_{m-1}\tau)]v^0\right\|_p.
    \end{aligned}
\end{equation}
Substituting \eqref{eq:L_1} and \eqref{eq:L_2} into \eqref{eq:L_1+L_2}, we get
\begin{equation}
\begin{aligned}
&\left\| [\mathcal{S}(\tau)-\mathcal{E}_A(\widetilde{a}_{m}\tau)]w^0-[\mathcal{S}(\tau)-\mathcal{E}_A(\widetilde{a}_{m}\tau)]v^0\right\|_p\\
\leq& \kappa b_m  \mathrm{e}^{\kappa \widetilde{b}_{m}\tau}\tau \left\|w^0-v^0\right\|_p +\left\|[\widetilde{\mathcal{S}}(\tau)-\mathcal{E}_A(\widetilde{a}_{m-1}\tau)]w^0-[\widetilde{\mathcal{S}}(\tau)-\mathcal{E}_A(\widetilde{a}_{m-1}\tau)]v^0\right\|_p\\
\leq& \cdots \leq  \kappa( b_m\mathrm{e}^{\kappa \widetilde{b}_{m}\tau}+ b_{m-1}\mathrm{e}^{\kappa \widetilde{b}_{m-1}\tau}+\cdots+b_1\mathrm{e}^{\kappa \widetilde{b}_{1}\tau})\tau \left\|w^0-v^0\right\|_p\\
\leq&\kappa(b_m+b_{m-1}+\cdots+b_1)\mathrm{e}^{\kappa \widetilde{b}\tau}\tau \left\|w^0-v^0\right\|_p
=\kappa \widetilde{b}\mathrm{e}^{\kappa \widetilde{b}\tau}\tau \left\|w^0-v^0\right\|_p.
\end{aligned}
\end{equation}
Furthermore, let $\eta = |\Omega|^{\frac{1}{p}}|f(0)|$. From \eqref{eq:S-E_A}, applying  \eqref{lemma:E_A_stability}, \eqref{lemma:E_B-I_stability}, and \eqref{lemma:E_B_stability}, we obtain
\begin{equation}
\begin{aligned}
&\left\| \mathcal{S}(\tau)w^0-\mathcal{E}_A(\widetilde{a}_{m}\tau)w^0\right\|_p\\
\leq&  \left\|[\mathcal{E}_B(b_m\tau)-\mathcal{I}]\mathcal{E}_A(a_m\tau)\widetilde{\mathcal{S}}(\tau)w^0\right\|_p+\left\|\mathcal{E}_A(a_m\tau)[\widetilde{\mathcal{S}}(\tau)-\mathcal{E}_A(\widetilde{a}_{m-1}\tau)]w^0\right\|_p\\
\leq& \left(\mathrm{e}^{\kappa b_m\tau}-1\right)  \left(\left\|\widetilde{\mathcal{S}}(\tau)w^0\right\|_p+\frac{\eta}{\kappa}\right)+\left\|\widetilde{\mathcal{S}}(\tau)w^0-\mathcal{E}_A(\widetilde{a}_{m-1}\tau)w^0\right\|_p\\
\leq& \left(\mathrm{e}^{\kappa b_m\tau}-1\right)  \left(\left\|\left[\widetilde{\mathcal{S}}(\tau)-\mathcal{E}_A(\widetilde{a}_{m-1}\tau)\right]w^0\right\|_p+\left\|w^0\right\|_p+\frac{\eta}{\kappa}\right)+\left\|\left[\widetilde{\mathcal{S}}(\tau)-\mathcal{E}_A(\widetilde{a}_{m-1}\tau)\right]w^0\right\|_p\\
=& \mathrm{e}^{\kappa b_m\tau} \left\|\left[\widetilde{\mathcal{S}}(\tau)-\mathcal{E}_A(\widetilde{a}_{m-1}\tau)\right]w^0\right\|_p + \left(\mathrm{e}^{\kappa b_m\tau}-1\right) \left(\left\|w^0\right\|_p+\frac{\eta}{\kappa}\right)\\
\leq& \cdots\leq \mathrm{e}^{\kappa (b_m+b_{m-1}+\cdots+b_1)\tau} \left\|w^0 - w^0\right\|_p + \left(\mathrm{e}^{\kappa (b_m+b_{m-1}+\cdots+b_1)\tau}-1\right)\left(\left\|w^0\right\|_p+\frac{\eta}{\kappa}\right)\\
\leq&\left(\mathrm{e}^{\kappa\widetilde{b}\tau}-1\right) \left(\left\|w^0\right\|_p+\frac{\eta}{\kappa}\right),
\end{aligned}
\end{equation}
which yields the expected result.
\end{proof}

We now conclude the $L^p$-norm stability of MPE splitting operators for $1\leq p\leq\infty$ with the following theorem. 

\begin{theorem}[Stability of $\mathcal{S}^{[k]}(\tau)$]\label{thm:stability_S^k}Suppose that  $w^0,v^0\in L^p(\Omega)$ for $1\leq p\leq \infty$, and \eqref{assumption:f'_bounded} is satisfied. For MPE splitting operator $\mathcal{S}^{[k]}(\tau)$ 
 defined in \eqref{eq:MPE_scheme} with $a_{i,j}\geq 0$, $b_{i,j}\geq 0$,  $\sum^{M}_{i=1}c_i = 1$, and $\sum_{j=1}^{m_i} a_{i,j} = 1$ for \(1\leq i\leq M\), it holds that 
    \begin{equation}\label{eq:stability_S^k}
        \left\|\mathcal{S}^{[k]}(\tau)w^0-\mathcal{S}^{[k]}(\tau)v^0\right\|_p\leq \mathrm{e}^{C\tau}\left\|w^0-v^0\right\|_p,\quad\forall \tau>0,
    \end{equation}
where $C=(\widetilde{c}+1)\kappa b$, $\widetilde{c} = \sum^{M}_{i=1}|c_i|$, and $b = \max_{1\leq i\leq M} \left\{\sum_{j=1}^{m_i} b_{i,j}\right\}$.
Moreover, for any $T>0$ and $0<n\tau\leq T$, the Sobolev bounds of $\mathcal{S}^{[k]}(\tau)$ can be derived as
    \begin{equation}
\left\|\left[\mathcal{S}^{[k]}(\tau)\right]^nw^0\right\|_p\leq \mathrm{e}^{\widetilde{c}\kappa bT}\left(\left\|w^0\right\|_p+\frac{|\Omega|^{\frac{1}{p}}|f(0)|}{\kappa}\right),
    \end{equation}
where $|\Omega|^{\frac{1}{\infty}}=1$ as aforementioned.
\end{theorem}
\begin{remark}
    It should be noted that the order condition of MPE splitting inherently enforces the condition \( \sum_{i=1}^{M} c_i = 1 \) for any \( k \)th-order (\( k \geq 1 \)) MPE splitting operator, while the condition \( \sum_{j=1}^{m_i} a_{i, j}= 1 \) for \(1\leq i\leq M\) is imposed to obtain the stability result in Theorem \ref{thm:stability_S^k}.
\end{remark}
\begin{proof}
We first define $\widetilde{a}_{i,s}=\sum_{j=1}^sa_{i,j}$, $\widetilde{b}_{i,s}=\sum_{j=1}^sb_{i,j}$, $\widetilde{c} = \sum^{M}_{i=1}|c_i| $, and $ b = \max_i \widetilde{b}_{i,m_i}$. We reformulate the definition of $\mathcal{S}^{[k]}(\tau)$ in \eqref{eq:MPE_scheme} as
\begin{equation}
    \begin{aligned}
        \mathcal{S}^{[k]}(\tau) =\sum_i^{M}c_i \mathcal{S}_i^{[k]}(\tau),
    \end{aligned}
\end{equation}
where $\mathcal{S}_i^{[k]}(\tau)=\mathcal{E}_B(b_{i,m_i}\tau)\mathcal{E}_A(a_{i,m_i}\tau)\cdots\mathcal{E}_B(b_{i,1}\tau)\mathcal{E}_A(a_{i,1}\tau)$. 
We can directly obtain the following result by using condition $\sum^{M}_{i=1}c_i = 1$ together with \eqref{lemma:E_A_stability} and \eqref{lemma:S-E_A},
\begin{equation}
\begin{aligned}
&\left\|\mathcal{S}^{[k]}(\tau)w^0-\mathcal{S}^{[k]}(\tau)v^0\right\|_p\\
=&\left\|\sum^{M}_{i=1}c_i\left\{\left[ \mathcal{S}_i^{[k]}(\tau)-\mathcal{E}_A(\tau)\right]w^0-\left[\mathcal{S}_i^{[k]}(\tau)-\mathcal{E}_A(\tau)\right]v^0\right\}+\sum^{M}_{i=1}c_i \mathcal{E}_A(\tau)(w^0-v^0)\right\|_p\\
\leq& \sum^{M}_{i=1}|c_i| \kappa  \widetilde{b}_{i,m_i}e^{\kappa \widetilde{b}_{i,m_i}\tau}\tau\left\|w^0-v^0\right\|_p+\left\|w^0-v^0\right\|_p.\\
\end{aligned}
\end{equation}
Since $\widetilde{c}\geq1$ and $b=\max_i \widetilde{b}_{i,m_i}$, we have
\begin{equation}
\begin{aligned}
       \left\|\mathcal{S}^{[k]}(\tau)w^0-\mathcal{S}^{[k]}(\tau)v^0\right\|_p&\leq(\widetilde{c}\kappa b \mathrm{e}^{\kappa b\tau} \tau + 1) \left\|w^0-v^0\right\|_p\leq \mathrm{e}^{(\widetilde{c}+1)\kappa b \tau} \left\|w^0-v^0\right\|_p.
\end{aligned}
\end{equation}

Furthermore, let $\eta = |\Omega|^{\frac{1}{p}}|f(0)|$. By \eqref{lemma:E_A_stability} and \eqref{lemma:S-E_A} together with $\sum^{M}_{i=1}c_i=1$, we obtain
\begin{equation}
    \begin{aligned}
    \left\|\mathcal{S}^{[k]}(\tau)w^0\right\|_p&=\left\|\sum^{M}_{i=1}c_i\left[ \mathcal{S}_i^{[k]}(\tau)-\mathcal{E}_A(\tau)\right]w^0+\mathcal{E}_A(\tau)w^0\right\|_p\\
    &\leq \sum^{M}_{i=1}|c_i| \left(\mathrm{e}^{\kappa\widetilde{b}_{i,m_i}\tau}-1\right)  \left(\left\|w^0\right\|_p+\frac{\eta}{\kappa}\right)+\left\|w^0\right\|_p\leq \mathrm{e}^{\widetilde{c}\kappa b\tau}\left\|w^0\right\|_p+\left(\mathrm{e}^{\widetilde{c}\kappa b \tau}-1\right)\frac{\eta}{\kappa}.
    \end{aligned}
\end{equation}
Therefore, we can establish Sobolev bounds on the numerical solution $u^n$ at $t
_n$-level as follows:
\begin{equation}
\begin{aligned}
     \left\|\left[\mathcal{S}^{[k]}(\tau)\right]^nw^0\right\|_p &\leq\mathrm{e}^{\widetilde{c}\kappa b\tau}\left\|\left[\mathcal{S}^{[k]}(\tau)\right]^{n-1}w^0\right\|_p+\left(\mathrm{e}^{\widetilde{c}\kappa b \tau}-1\right)\frac{\eta}{\kappa}\leq \cdots\\
    &\leq\mathrm{e}^{\widetilde{c}\kappa bn\tau}\left\|w^0\right\|_p+\left(\mathrm{e}^{\widetilde{c}\kappa b (n-1)\tau}+\mathrm{e}^{\widetilde{c}\kappa b (n-2)\tau}+\cdots+1\right)\left(\mathrm{e}^{\widetilde{c}\kappa b \tau}-1\right)\frac{\eta}{\kappa}\\
    &\leq\mathrm{e}^{\widetilde{c}\kappa bT}\left\|w^0\right\|_p+\mathrm{e}^{\widetilde{c}\kappa b T}\frac{\eta}{\kappa},\\
\end{aligned}
  \end{equation}
which completes the proof.
\end{proof}

\begin{remark}
  The above stability analysis can also be extended to homogeneous Neumann boundary conditions, as the stability of the solution operators to both subproblems. However, for general Dirichlet, Robin, and Neumann boundary conditions, order reduction might happen even for Strang splitting (see for example \cite{einkemmer2015overcoming,einkemmer2016overcoming,bertoli2020strang}). It is still a challenging problem to investigate the stability of MPE splittings for these boundary conditions.
\end{remark}

\section{Convergence of MPE splitting methods for semilinear parabolic equations}\label{section:convergence} In this section, a rigorous local error analysis and convergence of
MPE splitting methods for semilinear parabolic equations are given. We first derive the Sobolev bounds of the exact solution to the semilinear parabolic equations. To derive the global temporal error estimate, we employ the formal calculus of Lie derivative and iterated Lie-commutators, which help to establish the consistency of MPE splitting methods. Finally, we establish the convergence results of MPE splitting methods based on the previous stability results.

\subsection{Stability of $\mathcal{E}(\tau)$}\label{subsection: stability of exact solution}
We now establish the Sobolev bounds of the exact solution to \eqref{eq:semilinear equation} that will be helpful for later convergence analysis.

\begin{theorem}[Sobolev bounds of exact solution]\label{thm:exact_sol_bound}
Suppose that the initial data $u^0 \in H^m_\text{per}(\Omega)$ with $m\geq 1$, $f \in C^{m}(\mathbb{R})$, and \eqref{assumption:f'_bounded} is satisfied.  Let $u(t)$ be the exact solution to \eqref{eq:semilinear equation} with periodic boundary conditions corresponding to the initial data $u^0$. Then, for given $T\geq 0$, the exact solution $u(t)$  possesses the Sobolev bounds:
\begin{equation}\label{eq:AC_exact_sol_HK_bound}
    \underset{0\leq t\leq T}{\sup}\left\|u(t)\right\|_{H^m}\leq  \mathrm{e}^{\kappa T}(\left\|u^0\right\|_{H^m}+C),
\end{equation}
     where $C\geq 0 $ depends on $(\|u^0\|_{H^{m-1}(\Omega)},d,\kappa,m,T)$.

     In the case of $u^0 \in L^p(\Omega)$ with $p\in[1,\infty]$, $u(t)$ also possesses the $L^p$-norm bounds:
     \begin{equation}\label{eq:AC_exact_sol_Lp_bound}
    \underset{0\leq t\leq T}{\sup}\left\|u(t)\right\|_{p}\leq  \mathrm{e}^{\kappa T}(\left\|u^0\right\|_{p}+C),
\end{equation}
     where $C\geq 0 $ depends on $(|\Omega|,T,|f(0)|)$.
\end{theorem}
\begin{proof}
The proof of \eqref{thm:exact_sol_bound} is shown in \eqref{appendix:proof of thm}.
\end{proof}

\subsection{Consistency of MPE splitting methods} \label{subsection:consistency of MPE} 
In this part, we will demonstrate that the $k$th-order MPE splitting method given by $\mathcal{S}^{[k]}(\tau)$ in \eqref{eq:MPE_scheme} possesses an $O(\tau^{k+1})$ local error, based on the use of Lie derivative and iterated Lie-commutators introduced in \eqref{appendix:lie derivative}.

A particularly useful tool for the theoretical error
analysis of higher-order exponential operator splitting methods is the calculus of Lie-derivatives. This calculus allows us to formally extend arguments for the less involved linear case to nonlinear problems. In \cite{thalhammer2008high}, the local error expansion of the splitting methods consisting of two linear operators $A,B$ is derived. We note that, by applying the formal calculus of Lie derivative, the expansion for the nonlinear case can be obtained by replacing the linear operators with the Lie derivative and reversing the
order. Next, we provide the following result on the expansion of the defect operator $\mathcal{D}(t)$ \eqref{eq:defect}. 

\begin{lemma}[Expansion of $\mathcal{D}(t)$]\label{lemma:Local error expansion}
Let $\widetilde{a}_{i,s}=\sum_{j=1}^{s}a_{i,j}$ for $1 \leq s \leq m_i$, $T_{s}=\{\sigma=\left(\sigma_{1}, \ldots, \sigma_{s}\right) \in \mathbb{R}^{s}: 0 \leq \sigma_{s} \leq \cdots \leq \sigma_{1} \leq \sigma_{0}=t\}$, and $L_{i,s}=\{\lambda=\left(\lambda_{1}, \ldots, \lambda_{s}\right) \in \mathbb{N}^{s}: 1 \leq \lambda_{s} \leq \cdots \leq \lambda_{1} \leq \lambda_{0}=m_i\}$. Provided that the conditions $\sum_{i=1}^Mc_i=1$ and  $\widetilde{a}_{i,m_i}=1$ for $1\leq i\leq M$ are satisfied, the defect operator \eqref{eq:defect} of the MPE splitting operator \eqref{eq:MPE_scheme} applied to the evolutionary problem \eqref{eq:evolutionary problem} admits the expansion
\begin{equation}\label{eq:local expansion}
    \begin{aligned}
        \mathcal{D}(t)v &=\sum_{q=1}^{k} \sum_{\substack{\mu \in \mathbb{N}^{q} \\
|\mu| \leq k-q}} \frac{1}{\mu!} t^{q+|\mu|} C_{q \mu} \prod_{\ell=1}^{q} \operatorname{ad}_{D_{A}}^{\mu_{\ell}}\left(D_{B}\right) \mathrm{e}^{t D_{A}} (v)+\mathcal{R}_{k+1}(t, v), \quad 0 \leq t \leq T,
    \end{aligned}
\end{equation}
\begin{equation}\label{eq:new_order_condition}
    \begin{aligned}
C_{q 
\mu}&=\sum_{i=1}^{M}c_i\sum_{\lambda \in L_{i,q}} \alpha_{\lambda} \prod_{\ell=1}^{q} b_{i,\lambda_{\ell}} \widetilde{a}_{i,\lambda_{\ell}}^{\mu_{\ell}}-\prod_{\ell=1}^{q} \frac{1}{\mu_{\ell}+\cdots+\mu_{q}+q-\ell+1}.
    \end{aligned}
\end{equation}

The remainder $\mathcal{R}_{k+1}(t, v)=\mathcal{R}_{k+1}^{(1)}(t, v)-\mathcal{R}_{k+1}^{(2)}(t, v)-\mathcal{R}_{k+1}^{(3)}(t, v)$ comprises the terms
\begin{equation}
    \begin{aligned}
\mathcal{R}_{k+1}^{(1)}(t, v)&=\int_{T_{k+1}} \mathrm{e}^{\sigma_{k+1} D_{A+B}} \prod_{j=1}^{k+1}\left(D_{B} \mathrm{e}^{\left(\sigma_{j-1}-\sigma_{j}\right) D_{A}}\right) (v) \mathrm{~d} \sigma,
\\
\mathcal{R}_{k+1}^{(2)}(t, v)&=t^{k+1} 
\sum_{i=1}^{M}c_i\sum_{j=1}^{k+1} \sum_{\lambda \in L_{i,k+1}} \widetilde{\alpha}_{j \lambda} \prod_{\ell=1}^{\lambda_{k+1}-1}\left(\mathrm{e}^{a_{i,\lambda_{k+1}-\ell} t D_{A}} \mathrm{e}^{b_{i,\lambda_{k+1}-\ell} t D_{B}}\right) \\
&\times \mathrm{e}^{a_{i,\lambda_{k+1}} t D_{A}} \varphi_{j}\left(b_{i,\lambda_{k+1}} t D_{B}\right) \prod_{\ell=1}^{k+1}\left(b_{i,\lambda_{\ell}} D_{B} \mathrm{e}^{\left(\widetilde{a}_{i,\lambda_{\ell-1}}-\widetilde{a}_{i,\lambda_{\ell}}\right) t D_{A}}\right) (v),
    \\
\mathcal{R}_{k+1}^{(3)}(t, v)&=\sum_{q=1}^{k}\left(\int_{T_{q}} r_{q, k-q+1}(\sigma, v) \mathrm{d} \sigma-t^{q} \sum_{i=1}^{M}c_i\sum_{\lambda \in L_{i,q}} \alpha_{\lambda} \prod_{\ell=1}^{q} b_{i,\lambda_{\ell}} r_{q, k-q+1}\left(\widetilde{a}_{i,\lambda} t, v\right)\right), \\
    r_{q, k-q+1}(\sigma, v)&=(k-q+1) \int_{0}^{1}(1-\zeta)^{k-q} \sum_{\substack{\mu \in \mathbb{N}^{q} \\
 |\mu|=k-q+1}}\frac{(-1)^{|\mu|}}{\mu!} \sigma^{\mu} \mathrm{e}^{\sigma_{q} \zeta D_{A}}\\
&\times \prod_{\ell=1}^{q}\left(\operatorname{ad}_{D_{A}}^{\mu_{\ell}}\left(D_{B}\right) \mathrm{e}^{\left(\sigma_{\ell-1}-\sigma_{\ell}\right) \zeta D_{A}}\right) (v) \mathrm{~d} \zeta, 
    \end{aligned}
\end{equation}
where $\varphi_j$ is defined in \eqref{eq:phi} and the coefficients $\alpha_{\lambda},\widetilde{\alpha}_{j\lambda}$ are computable by recurrence (see Maltab code for this computation at
\url{http://techmath.uibk.ac.at/mecht/MyHomepage/Research/recurrence.m}).
 \end{lemma}

\begin{proof}
   Comparing to the proof of the expansion of the defect operator of splitting operators in \cite{koch2013error}, this proof for MPE splitting operators only needs to take into account weights. We omit the details here.
\end{proof}

It is noteworthy that if the coefficients $C_{q\mu} = 0$ for $\mu \in M$ with $|\mu| \leq k - q$ and $1 \leq q \leq k$, the local error expansion of \eqref{eq:local expansion} simplifies to $\mathcal{D}(\tau)v = \mathcal{R}_{k+1}(\tau, v)  = O(\tau^{k+1})$.  In \cite{thalhammer2008high}, Thalhammer stated that for splitting methods, the classical order conditions coincide with conditions $C_{q\mu} = 0$ for $\mu \in M$ with $|\mu| \leq k - q$ and $1 \leq q \leq k$. One can also refer to \cite{blanes2013new} for more discussion on the order conditions of splitting methods. 

The remaining question is to determine the regularity requirements on $v$ to ensure that $\mathcal{R}_{k+1}(\tau, v)$ is bounded. Specifically, the regularity of $v$ required to estimate the iterated Lie-commutators $\operatorname{ad}_{D_{A}}^{i}\left(D_{B}\right)$ for $i=0,1,\ldots,k$ must be determined. We next state and prove the consistency of MPE splitting operators.

\begin{theorem}[Consistency for $\mathcal{S}^{[k]}(\tau)$]\label{thm:consistency of MPE}
Let $\mathcal{S}^{[k]}(\tau)$ be the $k$th-order MPE splitting operator defined in \eqref{eq:MPE_scheme} with $\sum_{i=1}^Mc_i=1$ and $\sum_{j=1}^{m_i}a_{i,j}=1$ for $1\leq i\leq M$, whose coefficients satisfy the order conditions $C_{q\mu} = 0$ for all $\mu \in \mathbb{N}^q$ with $|\mu| \leq k - q$ and $1 \leq q \leq k$ as well as $a_{i,j}\geq 0$, $b_{i,j}\geq 0$. Suppose that $v \in H^{m}_\text{per}(\Omega)$ with $m=k + {d}/{2}$, the spatial dimension $d=1,2,3$, $f \in C^{2k}(\mathbb{R})$, and \eqref{assumption:f'_bounded} is satisfied. Then, the local truncation error satisfies the following estimates:  
\begin{equation}
\begin{aligned}
    \left\|\mathcal{S}^{[k]}(\tau)v - \mathcal{E}(\tau)v\right\|_p &\leq C \tau^{k+1}, \quad \text{for} \quad p\in[1,\infty),
\end{aligned}
\end{equation}
where the constant $C > 0$ depends on ($\|v\|_{H^{m}}$, $\Omega$, $d$, $\kappa$, $k$).

In the case of $p=\infty$, for any $\delta>0$ and $u^0\in H^{m+\delta}_\text{per}(\Omega)$, 
\begin{equation}
\begin{aligned} 
    \left\|\mathcal{S}^{[k]}(\tau)v - \mathcal{E}(\tau)v\right\|_\infty \leq C \tau^{k+1},
\end{aligned}
\end{equation}
where the constant $C > 0$ depends on ($\|v\|_{H^{m+\delta}}$, $\Omega$, $d$, $\kappa$, $k$).
\end{theorem}

\begin{proof} 
From \eqref{lemma:Local error expansion}, since the
 $\mathcal{S}^{[k]}(\tau)$ splitting operator satisfies the order conditions $C_{q\mu} = 0$ for $\mu \in M$ with $|\mu| \leq k - q$ and $1 \leq q \leq k$, it suffices to consider the remainder $\mathcal{R}_{k+1}(t, v)$. Recall that $\mathrm{e}^{\tau D_{A}}\mathcal{I}(v) = \mathcal{E}_A(\tau)v=\mathrm{e}^{\tau\Delta}v$,  $\mathrm{e}^{\tau D_{B}}\mathcal{I}(v) = \mathcal{E}_B(\tau)v$ and $\mathrm{e}^{\tau D_{A+B}}\mathcal{I}(v) = \mathcal{E}(\tau)v$ for any $\tau\geq 0$. Using the stability of $\mathcal{E}_A(\tau)$, $\mathcal{E}_B(\tau)$, and $\mathcal{E}(\tau)$ in \eqref{lemma:E_A_stability}, \eqref{lemma:E_B_stability}, and \eqref{thm:exact_sol_bound}, respectively,  along with the smoothness of $f$ in $\mathbb{R}$, we conclude that the remainders $\mathcal{R}_{k+1}^{(1)}$ and $\mathcal{R}_{k+1}^{(2)}$ in Lemma \ref{lemma:Local error expansion} are bounded as follows:
 \begin{equation}
     \left\|\mathcal{R}_{k+1}^{(1)}(\tau, v)\right\|_p+\left\|\mathcal{R}_{k+1}^{(2)}(\tau, v)\right\|_p\leq C_1\tau^{k+1},
 \end{equation}
where $1\leq p\leq\infty$ and  $C_1$ depends on $\|v\|_p$ and derivatives of $f$. 
Regarding the remainder $\mathcal{R}_{k+1}^{(3)}(\tau, v)$, we need to estimate the bounds of the term 
\begin{equation}\label{eq:derivatives of g}
     \begin{aligned}
\mathrm{e}^{\sigma_{q} \zeta D_{A}}\prod_{\ell=1}^{q}\left(\operatorname{ad}_{D_{A}}^{\mu_{\ell}}\left(D_{B}\right) \mathrm{e}^{\left(\sigma_{\ell-1}-\sigma_{\ell}\right) \zeta D_{A}}\right)(v)
     \end{aligned}
 \end{equation}
 in $r_{q,k-q+1}(\sigma,v)$. 

 Firstly, we consider the iterated Lie-commutators $\operatorname{ad}^i_{A}(B)(w)$ for $i=0,1,\ldots,k$. Recall the convention that $\operatorname{ad}^0_{A}(B)(w) = B(w)=f(w)$ for $i=0$. The first iterated Lie-commutator is given by $\operatorname{ad}_{A}(B)w = A'(w)B(w) - B'(w)A(w)$.  Specifically, we have $A(w) = \Delta w$ and $B(w) = f(w)$,  with the  Fr\'{e}chet-derivatives given by $A'(w) = \Delta$ and $B'(w) = f'(w)$. A brief calculation yields $\operatorname{ad}_{A}(B)(w) = \Delta \left(f(w)\right) - f'(w)\Delta w$. To simplify our notations, we only need to consider $\operatorname{ad}_{A}(B)$ in the one-dimensional case. In this case, we have
\begin{equation} \label{eq:1st_commutator}
    \begin{gathered}
        \operatorname{ad}_{A}(B)(w) =  \left(f''(w)(\partial_x w)^2 +f'(w)\partial^2_x w\right) -f'(w)\partial^2_x w = f''(w)(\partial_x w)^2.
    \end{gathered}
\end{equation}
By induction, for $0<i\leq k$, $\operatorname{ad}^i_{A}(B)(w)$ involves the spatial derivatives of $w$ (up to $2i$ spatial derivatives are multiplied) and derivatives of $f$ (up to the order $2i$). Since $f^{(l)}$ is continuous in $\mathbb{R}$ and then bounded on a given bounded interval for $1 \leq l \leq 2i$, we now focus on specifying the regularity requirement for $v$ such that \eqref{eq:derivatives of g} remains bounded for the $k$th-order MPE splitting operators. This implies that we need to consider the highest-order spatial derivatives of $v$, since lower-order terms can be estimated via Sobolev embeddings.

It is important to note that the first Lie-commutator involves the first-order spatial derivative term $(\partial_x w)^2$. In fact, we assume the terms involve the highest spatial derivative of $w$ of the form $G(w) = F(w)\left(\partial_x^iw\right)^l$ for positive integers $l,i$, where $F(w)$ involves the derivatives of $w$ of the order up to $i-1$. Using the Fr\'{e}chet-derivatives, it follows that
\begin{equation}\label{eq:ad_general}
    \begin{aligned}    \operatorname{ad}_{A}(G)(w) &=  A'(w)G(w) - G'(w)A(w)\\
    & =\partial_x^2\left(F(w)\left(\partial_x^iw\right)^l\right) - \left(\partial_x^iw\right)^lF'(w)\left(\partial_x^2w\right)-lF(w)\left(\partial_x^iw\right)^{l-1}\partial_x^{i+2}w.
    \end{aligned}
\end{equation}
A brief calculation for the first term of \eqref{eq:ad_general} yields that
\begin{equation}
    \begin{aligned}
\partial_x^2\left(F(w)\left(\partial_x^iw\right)^l\right) =& \partial_x\left(\left(\partial_x^iw\right)^lF'(w)(\partial_xw)+lF(w)\left(\partial_x^iw\right)^{l-1}\partial_x^{i+1}w\right)\\
=& \left(\partial_x^iw\right)^{l}F''(w)\left(\partial_xw,\partial_xw\right) + \left(\partial_x^iw\right)^{l}F'(w)(\partial^2_xw)\\
&+2l\left(\partial_x^iw\right)^{l-1}\partial_x^{i+1}wF'(w)(\partial_xw)+l(l-1)F(w)\left(\partial_x^iw\right)^{l-2}\left(\partial_x^{i+1}w\right)^2\\
&+lF(w)\left(\partial_x^iw\right)^{l-1}\partial_x^{i+2}w.
    \end{aligned}
\end{equation}
Thus, we obtain that
\begin{equation}
    \begin{aligned}
\operatorname{ad}_{A}(G)(w) =&\left(\partial_x^iw\right)^{l}F''(w)\left(\partial_xw,\partial_xw\right) +2l\left(\partial_x^iw\right)^{l-1}\partial_x^{i+1}wF'(w)(\partial_xw)\\
&+l(l-1)F(w)\left(\partial_x^iw\right)^{l-2}\left(\partial_x^{i+1}w\right)^2.
    \end{aligned}
\end{equation}
Note that the terms associated with $\partial_x^{i+2}w$ are canceled. Since the Fr\'{e}chet-derivatives $F'(w)(\partial_xw),F''(w)(\partial_xw,\partial_xw)$ at most involve $i$th-order spatial derivatives of $w$, we can conclude that $\operatorname{ad}^{i+1}_{A}(B)(w)$ involves spatial derivatives only up to the order $i+1$ by induction. The extension to arbitrary finite space dimensions is more technical, but otherwise straightforward.

Recalling \eqref{eq:ad}, the term \eqref{eq:derivatives of g}
\begin{equation}
\begin{aligned}
\mathrm{e}^{\sigma_{q} \zeta D_{A}}\prod_{\ell=1}^{q}\left(\operatorname{ad}_{D_{A}}^{\mu_{\ell}}\left(D_{B}\right) \mathrm{e}^{\left(\sigma_{\ell-1}-\sigma_{\ell}\right) \zeta D_{A}}\right)(v)=(-1)^{|\mu|}\mathrm{e}^{\sigma_{q} \zeta D_{A}}\prod_{\ell=1}^{q}\left(D_{\operatorname{ad}_{{A}}^{\mu_{\ell}}\left({B}\right) }\mathrm{e}^{\left(\sigma_{\ell-1}-\sigma_{\ell}\right) \zeta D_{A}}\right)(v)
\end{aligned}
\end{equation}
 in $r_{q,k-q+1}(\sigma,v)$ only contains $f$ and the iterated Lie-commutators $\operatorname{ad}^i_{A}(B)$ and their derivatives. Let $G_j(v) = \prod_{\ell=1}^{j}\left(D_{\operatorname{ad}_{{A}}^{\mu_{\ell}}\left({B}\right) }\mathrm{e}^{\left(\sigma_{\ell-1}-\sigma_{\ell}\right) \zeta D_{A}}\right)(v)$, for $j=1,2,\ldots,q$. Then, we have
 \begin{equation}
     \begin{aligned}
         &G_1(v) = \mathrm{e}^{\left(\sigma_{0}-\sigma_{1}\right) \zeta \Delta}\operatorname{ad}_{{A}}^{\mu_{1}}\left({B}\right)(v),\\
         &G_j(v) = \left(D_{\operatorname{ad}_{{A}}^{\mu_{j}}\left({B}\right) }\mathrm{e}^{\left(\sigma_{j-1}-\sigma_{j}\right) \zeta D_{A}}\right) G_{j-1}(v) = G'_{j-1}(\mathrm{e}^{\left(\sigma_{j-1}-\sigma_{j}\right) \zeta \Delta}v)\mathrm{e}^{\left(\sigma_{j-1}-\sigma_{j}\right) \zeta \Delta}\operatorname{ad}_{{A}}^{\mu_{j}}\left({B}\right)(v),\\
        &\mathrm{e}^{\sigma_{q} \zeta D_{A}}\prod_{\ell=1}^{q}\left(\operatorname{ad}_{D_{A}}^{\mu_{\ell}}\left(D_{B}\right) \mathrm{e}^{\left(\sigma_{\ell-1}-\sigma_{\ell}\right) \zeta D_{A}}\right)(v) = \mathrm{e}^{\sigma_{q} \zeta D_{A}}  G_q(v) = G_q(\mathrm{e}^{\sigma_{q} \zeta \Delta}v),  \end{aligned}
 \end{equation}
where integer $1<j\leq q$. From the above relations we see that the order of the highest-order spatial derivatives of $v$ in $G_j(v)$ is $\sum_{\ell=1}^j\mu_\ell$. The highest-order spatial derivatives of $v$ in the term \eqref{eq:derivatives of g} involve $\partial_x^{|\mu|}v$.

 Secondly, we consider the estimate of the $L^p$-norm of $\mathcal{R}_{k+1}^{(3)}(\tau, v)$. Firstly, we take into account the case of $p\in[1,\infty)$.  The $L^p$-norm of the term \eqref{eq:derivatives of g} can be bounded by in terms of the ${H^{|\mu|+\frac{d}{2}}}$ norm of $v$ using the stability of $\mathcal{E}_A(\tau)$ in Lemma \ref{lemma:E_A_stability}, the smoothness of $f$, and the Sobolev embedding theorem. To be specific, for the spatial dimension $d=1,2,3$,  it follows that the dominant terms in the term \eqref{eq:derivatives of g} require $v\in H_\text{per}^{|\mu|+\frac{d}{2}}(\Omega)$ for the $L^p$-norm with $p\in[1,\infty)$, as $H_\text{per}^{|\mu|+\frac{d}{2}}(\Omega)\hookrightarrow W_\text{per}^{|\mu|,p}(\Omega)$ for $p\in[1,\infty)$ by the Sobolev embedding theorem. Hence, we obtain 
\begin{equation}
    \begin{aligned}
        \|r_{q,k-q+1}(\sigma,v)\|_p&\leq C_2 \int_{0}^{1} \sum_{\substack{\mu \in \mathbb{N}^{q} \\
 |\mu|=k-q+1}}\frac{\sigma^{\mu}}{\mu!}  \left\|\mathrm{e}^{\sigma_{q} \zeta D_{A}}\prod_{\ell=1}^{q}\left(\operatorname{ad}_{D_{A}}^{\mu_{\ell}}\left(D_{B}\right) \mathrm{e}^{\left(\sigma_{\ell-1}-\sigma_{\ell}\right) \zeta D_{A}}\right)v\right\|_p  \mathrm{~d} \zeta\\
 &\leq \sum_{\substack{\mu \in \mathbb{N}^{q} \\
 |\mu|=k-q+1}}C_{\mu,{\|v\|_{H^{|\mu|+\frac{d}{2}}}}}\frac{\sigma^\mu}{\mu!},
    \end{aligned}
\end{equation}
where $C_2,C_{\mu,{\|v\|_{H^{|\mu|+\frac{d}{2}}}}}$ are constants. Using the fact that
    \begin{equation}
        \begin{aligned}
            \int_{T_{q}} \sigma^{\mu} \mathrm{d} \sigma = t^{q+|\mu|}\prod^q_{l=1}\frac{1}{\mu_l+\cdots+\mu_q+q-l+1},
        \end{aligned}
    \end{equation}
    we have
     \begin{equation}
        \begin{aligned}
\mathcal{R}_{k+1}^{(3)}(\tau, v)
\leq&\sum_{q=1}^{k}\left(\int_{T_{q}} \|r_{q, k-q+1}(\sigma, v)\|_p \mathrm{d} \sigma+t^{q} \sum_{i=1}^{M}|c_i|\sum_{\lambda \in L_{i,q}} \alpha_{\lambda} \prod_{\ell=1}^{q} b_{i,\lambda_{\ell}} \|r_{q, k-q+1}\left(\widetilde{a}_{i,\lambda} t, v\right)\|_p\right)\\
            \leq& C_{k,{\|v\|_{H^{k+\frac{d}{2}}}}}t^{k+1},
        \end{aligned}
    \end{equation}   
where $C_{k,{\|v\|_{H^{k+\frac{d}{2}}}}}$ is a constant. Therefore, for MPE splitting operators $\mathcal{S}^{[k]}(\tau)$ of order $k$, we require $v\in H_\text{per}^{m}(\Omega)$ with $m= k+\frac{d}{2}$ for $L^p$-norm with $p\in[1,\infty)$. In a similar way, we require $v\in H_\text{per}^{m+\delta}(\Omega)$ with $\delta>0$ for the $L^\infty$-norm, since $H_\text{per}^{m+\delta}(\Omega)\hookrightarrow W_\text{per}^{k,\infty}(\Omega)$.
 \end{proof}

\subsection{Convergence results}\label{subsection:stability results}
By leveraging the previously established stability and consistency of MPE splitting operators, along with the regularity of the exact solution, we derive the following convergence result.

\begin{theorem}[Convergence of $\mathcal{S}^{[k]}(\tau)$]
    Assume that the coefficients of
the $k$th-order MPE splitting operator $\mathcal{S}^{[k]}(\tau)$ defined in \eqref{eq:MPE_scheme} satisfy the conditions $\sum_{i=1}^Mc_i=1$ and $\sum_{j=1}^{m_i}a_{i,j}=1$ for all $i$ and $a_{i,j},b_{i,j}\geq 0 $. Suppose that the initial data $u^0 \in H^{m}_\text{per}(\Omega)$ with $m = k+\lceil d/2\rceil$ and the spatial dimension $d=1,2,3$, $f \in C^{2k}(\mathbb{R})$, and \eqref{assumption:f'_bounded} is satisfied.  Let $u(t)$ be the exact solution to \eqref{eq:semilinear equation} with periodic boundary conditions corresponding to the initial data $u^0$. For any $T> 0$ and $0 < \tau \leq T$, the following error estimate holds:
 \begin{equation}
 \begin{aligned}
         &\sup_{n\tau \leq T}\left\|[\mathcal{S}^{[k]}(\tau)]^nu^0-u(n\tau)\right\|_p\leq  C\tau^k,\quad \text{for} \quad p\in[1,\infty),
 \end{aligned}
 \end{equation}
 where  $C> 0$ depends on $(\|u^0\|_{H^{m}},\Omega,d,\kappa,k,T)$.
 
 In the case of $p=\infty$, for any $\delta>0$ and $u^0\in H^{m+\delta}_\text{per}(\Omega)$, 
 \begin{equation}
 \begin{aligned}
         &\sup_{n\tau \leq T}\left\|[\mathcal{S}^{[k]}(\tau)]^nu^0-u(n\tau)\right\|_p\leq  C\tau^k,
 \end{aligned}
 \end{equation}
where  $C> 0$ depends on $(\|u^0\|_{H^{m+\delta}},\Omega,d,\kappa,k,T)$.
\end{theorem}
\begin{proof}
Consider the case of $p\in[1,\infty)$. Based on the Sobolev norm bounds of the exact solution $u(n\tau)$ in \eqref{thm:exact_sol_bound}, the smoothness of $f$, and the consistency in \eqref{thm:consistency of MPE}, we obtain the following inequality at $t
_n$-level ($t_n=n\tau$)
\begin{equation}
     \sup_{0\leq n\leq N}\left\| \mathcal{S}^{[k]}(\tau) u(t_n) - \mathcal{E}(\tau) u(t_n)\right\|_p \leq  C_1\tau^{k+1},\quad t_N\leq T,
\end{equation}
where $C_1$ depends on $(\|u^0\|_{H^{m}},\Omega,d,\kappa,k,T)$.
 
    Therefore, according to Theorem \ref{thm:stability_S^k}, we have
    \begin{equation}
    \begin{aligned}
\left\|u^{n}-u(t_n)\right\|_p \leq &  \left\| \mathcal{S}^{[k]}(\tau) u^{n-1} - \mathcal{S}^{[k]}(\tau) u(t_{n-1})\right\|_p +\left\| \mathcal{S}^{[k]}(\tau) u(t_{n-1}) - \mathcal{E}(\tau) u(t_{n-1})\right\|_p \\
    \leq & \mathrm{e}^{C_2\tau}\left\| u^{n-1} - u(t_{n-1})\right\|_p + C_1\tau^{k+1}\\
    \leq & \cdots \leq \mathrm{e}^{C_2T}\left\| u^{0} - u(0)\right\|_p + C_1\tau^{k+1} \frac{\mathrm{e}^{C_2T}-1}{\mathrm{e}^{C_2\tau}-1}\\
    \leq &  \frac{C_1}{C_2}(\mathrm{e}^{C_2T}-1)\tau^k = O(\tau^k),
    \end{aligned}
\end{equation}
 where $C_2$ represents the constant $C$ in \eqref{eq:stability_S^k}. This leads to the desired result. It is similar for the $L^\infty$-norm.
\end{proof}

\section{Numerical experiments}\label{section:numerical experiments}
In this section, we present two-dimensional numerical experiments to evaluate the stability, convergence order, and efficiency of the third-, fourth- and sixth-order operator splitting methods applied to semilinear parabolic models with periodic boundary conditions. The spectral method is utilized to compute the linear evolution operator $\mathcal{E}_A(\tau)$ using the FFT technique within the framework of a spectral method \cite{shen2011spectral}.  Given the spatial grid step $h=L/N$, the solution is computed as follows:
\begin{equation}\label{eq:sdg_linear}
    \begin{aligned}
    \mathcal{E}_A(\tau)v = \mathcal{F}_N^{-1}\{\exp(-\tau\lambda_{pq})\mathcal{F}_N[v](p,q)\}\quad \text{with}\quad 
    \lambda_{pq} = \left(\frac{2p\pi}{L}\right)^2 + \left(\frac{2q\pi}{L}\right)^2,
    \end{aligned}
\end{equation}
for $p,q = 0,1,\ldots,N/2-1, 0, -N/2+1,-N/2+2,\ldots,-1$. Here, $\mathcal{F}_N,\mathcal{F}_N^{-1}$ represent the FFT and IFFT algorithms, respectively.

To approximate $\mathcal{E}_B(\tau)$, we employ the explicit 10-stage, 4th-order strong stability preserving Runge–Kutta (SSP-RK) method \cite{gottlieb2011strong},  ensuring sufficient numerical accuracy without compromising the experimental results.

In the following examples, 
we ensure that \eqref{assumption:f'_bounded} holds for 
$f$ by employing the truncation technique proposed in \cite{shen2010numerical}. Specifically, we modify $f$ by truncating it for $|x| > M$ and replacing it with a linear polynomial that connects to the inner part, where $M$ is chosen to be sufficiently large. This guarantees the existence of a constant $\kappa$ satisfying \eqref{eq:f'_bounded}. 

\subsection{Convergence  tests of MPE splittings}\label{subsection:toy}
To test the accuracy of the proposed MPE splitting methods  in Section \ref{section:MPE for rd equation}, we first consider a toy model of the form:
\begin{equation}\label{eq:toy}
    \begin{aligned}
        u_{t}  = \varepsilon^2 \Delta u + f_{\rm T}(u), \\
    \end{aligned}
\end{equation}
where $\varepsilon = 0.1$ and $ f_{\rm T}(u) = \lambda \tanh{u}$ inherently satisfies the condition \eqref{eq:f'_bounded}  with a constant $\lambda$.  Note that $\mathcal{E}_B(\tau)v = \mathrm{arcsinh}((\sinh{v})\exp(\lambda \tau))$, which implies that $\mathcal{E}_B(\tau)v$ can be computed exactly. The computational domain is a $2\pi$-periodic square $\Omega = [0, 2\pi]^2$, discretized using a $1024 \times 1024$ grid. The initial data is set as $u^0(x, y) = 0.5 \sin(x) \sin(y)$. We employ $\mathcal{S}^{[6]}$ splitting with a small splitting step $\tau = 1/200$ to generate a reference solution at $T = 6$. Subsequently, numerical solutions are obtained by using various splitting steps at $T = 6$.

The $L^\infty$-errors between these numerical solutions and the reference solution are summarized in \eqref{tab:convergence_toy}. The convergence rates are approximately 3 for the $\mathcal{S}^{[3],1},\mathcal{S}^{[3],2}$ splittings, aligning with their third-order temporal accuracy. Furthermore, the $\mathcal{S}^{[4],1},\mathcal{S}^{[4],2},\mathcal{S}^{[4],3},$ and $\mathcal{S}^{[4],4}$ splittings demonstrate fourth-order accuracy in time. However, due to machine error and the rapid convergence of the sixth-order splitting, it is challenging to measure the convergence rate accurately very close to 6.

\begin{table}[htb!]
\renewcommand\arraystretch{1.}
\begin{center}
\def\temptablewidth{1\textwidth}
\caption{$L^\infty$-errors of numerical solutions of several splitting methods at time $T = 6$ to the toy model \eqref{eq:toy} for different splitting steps.}\label{tab:convergence_toy}
{\rule{\temptablewidth}{1.1pt}}
\begin{tabular*}{\temptablewidth}{@{\extracolsep{\fill}}ccccccc}
   $\tau$  & $\frac1{5}$  &$\frac1{10}$     &$\frac1{20}$     &$\frac1{40}$
   &$\frac1{80}$   \\  \hline
  Error of $\mathcal{S}^{[3],1}$  & $11599\times 10^{-3}$  & $1.6097\times 10^{-4}$   & $2.1071\times 10^{-5}$   & $7.3997\times 10^{-4}$ & $2.6925\times 10^{-6}$   \\[3pt]
Rate  of $\mathcal{S}^{[3],1}$   & -- & $2.8491$ & $2.9334$  & $2.9683$  &  $2.9849$ \\[3pt]\hline
  Error of $\mathcal{S}^{[3],2}$  & $79221\times 10^{-3}$  & $1.0234\times 10^{-3}$   & $1.2885\times 10^{-4}$   & $1.6124\times 10^{-5}$ & $2.0153\times 10^{-6}$   \\[3pt]
Rate  of $\mathcal{S}^{[3],2}$   & -- &  $2.9524$ & $2.9897$    $2.9985$  &  $3.0001$   \\[3pt]\hline
  Error of $\mathcal{S}^{[4],1}$  & $6.6385\times 10^{-4}$  & $4.6578\times 10^{-5}$   & $3.0387\times 10^{-6}$   & $1.9216\times 10^{-7}$ & $1.2018\times 10^{-8}$   \\[3pt]
Rate  of $\mathcal{S}^{[4],1}$   & -- &  $3.8332$ & $3.9381$  & $ 3.9831$  &  $3.9991$   \\[3pt]\hline
  Error of $\mathcal{S}^{[4],2}$    & $6.2021\times 10^{-5}$   & $4.4917\times 10^{-6}$   & $3.0276\times 10^{-7}$ & $1.9571\times 10^{-8}$ & $1.2295\times 10^{-9}$  \\[3pt]
Rate  of $\mathcal{S}^{[4],2}$   & -- & $3.7874$ &$3.8910$  &$3.9514$& $3.9926$ \\[3pt]\hline
  Error of $\mathcal{S}^{[4],3}$    & $1.6849\times 10^{-4}$   & $1.1994\times 10^{-5}$   & $7.8141\times 10^{-7}$ & $4.9444\times 10^{-8}$ & $3.2300\times 10^{-9}$  \\[3pt]
Rate  of $\mathcal{S}^{[4],3}$   & -- & $3.8122$ &$3.9401$  &$3.9822$& $3.9362$ \\[3pt]\hline
      Error of $\mathcal{S}^{[4],4}$    & $6.4011\times 10^{-4}$   & $4.5313\times 10^{-5}$   & $2.9559\times 10^{-6}$ & $1.8682\times 10^{-7}$ & $1.1853\times 10^{-8}$  \\[3pt]
Rate  of $\mathcal{S}^{[4],4}$   & -- &  $3.8203$  &  $3.9383$  &  $3.9839$  &  $3.9783$ \\[3pt]\hline
   $\tau$   &$\frac1{4}$     &$\frac1{6}$     &$\frac1{10}$
   &$\frac3{40}$  & $\frac1{20}$ \\[3pt]\hline
     Error of $\mathcal{S}^{[6]}$    & $3.6941\times 10^{-5}$   & $2.7087\times 10^{-6}$   & $3.0025\times 10^{-7}$ & $5.9464\times 10^{-8}$ & $5.8275\times 10^{-9}$  \\[3pt]
Rate  of $\mathcal{S}^{[6]}$   & -- &  $5.1150$   & $5.4249$  &  $5.6286$  & $5.7287$ 
\end{tabular*}
{\rule{\temptablewidth}{1pt}}
\end{center}
\end{table}

\subsection{Comparison of different splitting methods}\label{subsection:AC}
We consider the following Allen--Cahn (AC) equation with double-well potential \cite{quan2022stability}:
\begin{equation}\label{eq:AC}
    \begin{aligned}
        u_{t}  = \varepsilon^2 \Delta u +  f_{\rm AC}(u), \quad f_{\rm AC}(u) = u-u^3,\\
    \end{aligned}
\end{equation}
with the initial condition consisting of seven circles with centers and radii given in \eqref{tab:init_AC}:
\begin{equation} \label{eq:init_AC}
    \begin{aligned} 
        u^0(x, y)=-1+\sum_{i=1}^7 f_0\left(\sqrt{\left(x-x_i\right)^2+\left(y-y_i\right)^2}-r_i\right)
    \end{aligned}, 
\end{equation}
and $f_0$ is defined by
\begin{equation}
    \begin{aligned}
f_0(s)= \begin{cases}2 \mathrm{e}^{-\varepsilon^2 / s^2}, & \text { if } s<0, \\ 0, & \text { otherwise. }\end{cases}
    \end{aligned}
\end{equation}
To ensure the boundedness of derivative, $f_{\rm AC}$ is replaced with $\widetilde{f}_{\rm AC}$ (see \cite{shen2010numerical}):
\begin{equation}\label{eq:truncation_ac}
    \begin{aligned}
        \widetilde{f}_{\rm AC} = \begin{cases} 
(1-3M^2)u+2M^3,&u>M,\\
u-u^3, &u\in[-M,M],  \\
(1-3M^2)u-2M^3,&u<-M
\end{cases}
    \end{aligned}
\end{equation}
with $M=6$.

\begin{table}[htb!]
\renewcommand\arraystretch{1.0}
\begin{center}
\def\temptablewidth{1\textwidth}
\caption{Centers $(x_i,y_i)$ and radii $r_i$in the initial condition \eqref{eq:init_AC}.}\label{tab:init_AC}
{\rule{\temptablewidth}{1.1pt}}
\begin{tabular*}{\temptablewidth}{@{\extracolsep{\fill}}l|lllllll}
$i$ & 1 & 2 & 3 & 4 & 5 & 6 & 7 \\
\hline$x_i$ & $\pi / 2$ & $\pi / 4$ & $\pi / 2$ & $\pi$ & $3 \pi / 2$ & $\pi$ & $3 \pi / 2$ \\
$y_i$ & $\pi / 2$ & $3 \pi / 4$ & $5 \pi / 4$ & $\pi / 4$ & $\pi / 4$ & $\pi$ & $3 \pi / 2$ \\
$r_i$ & $\pi / 5$ & $2 \pi / 15$ & $2 \pi / 15$ & $\pi / 10$ & $\pi / 10$ & $\pi / 4$ & $\pi / 4$ \\
\end{tabular*}
{\rule{\temptablewidth}{1pt}}
\end{center}
\end{table}

In this numerical experiment, we use $512\times 512$ Fourier modes for the space discretization on computation domain $\Omega = [0,2\pi]^2$. The parameter is set as $\varepsilon = 0.1$. When $\|v\|_\infty \leq M$, we can straightforwardly use the explicit expressions to compute $\mathcal{E}_B(\tau)v$:
\begin{equation}
    \begin{aligned}
        \mathcal{E}_B(\tau)v = \frac{\mathrm{e}^{\tau}v}{\sqrt{1+(e^{2\tau}-1)v^2}}.
    \end{aligned}
\end{equation}

To compare the stability of the high-order MPE splitting methods with that of high-order splitting methods, we simulate the evolution of the AC model up to the process at $T= 10$ with $\tau = 1/40$  using the second-order Strang splitting $\mathcal{S}^{[2],1}$ in \eqref{eq:Strang_scheme}, the fourth-order MPE $\mathcal{S}^{[4],4}$ in \eqref{eq:richardson schemes},the fourth-order  splitting method admitting negative step coefficients $\mathcal{S}^{[4]}_{\rm neg}$ defined as:
\begin{equation}\label{eq:S4_neg}
    \begin{aligned}
        \mathcal{S}^{[4]}_{\rm neg}\left(\tau\right) =\mathrm{e}^{\frac{s}{2}\tau D_A} \mathrm{e}^{{s}\tau D_B}\mathrm{e}^{\frac{1-s}{2}\tau D_A}\mathrm{e}^{{(1-2s)}\tau D_B}\mathrm{e}^{\frac{1-s}{2}\tau D_A}\mathrm{e}^{{s}\tau D_B}\mathrm{e}^{\frac{s}{2}\tau D_A},
    \end{aligned}
\end{equation}
with $s = ({2+2^{\frac{1}{3}}+2^{-\frac{1}{3}}})/{3}$,  and an alternative 4th-order splitting scheme $\mathcal{S}^{[4]}_{\rm com}$ with complex time step coefficients given by (see \cite{castella2009splitting})
\begin{equation}\label{eq:S4_com}
    \begin{aligned}
        \mathcal{S}^{[4]}_{\rm com}\left(\tau\right) =\mathrm{e}^{\frac{r_1}{2}\tau D_A} \mathrm{e}^{{r_1}\tau D_B}\mathrm{e}^{\frac{r_1+r_2}{2}\tau D_A}\mathrm{e}^{{r_2}\tau D_B}\mathrm{e}^{\frac{r_1+r_2}{2}\tau D_A}\mathrm{e}^{{r_1}\tau D_B}\mathrm{e}^{\frac{r_1}{2}\tau D_A},
    \end{aligned}
\end{equation}
    where
        \[ r_1 = \frac{\mathrm{e}^{\mathrm{i} \pi /3}}{2\mathrm{e}^{\mathrm{i} \pi /3}+2^{1/3}},\quad r_2 = \frac{2^{1/3}}{2\mathrm{e}^{\mathrm{i} \pi /3}+2^{1/3}}\]
have positive real parts.
It is well-known that the exact solution to the AC equation with periodic boundary conditions satisfies the energy dissipation law: 
\begin{equation}\label{eq:AC_energy}
    \begin{aligned}
E(u) = \int_{\Omega} \frac{\varepsilon^2}{2}|\nabla u|^2 + \frac{(u^2-1)^2}{4} \dx,    
    \end{aligned}
\end{equation}
which decays with respect to time, and the maximum principle: if $\|u^0\|_\infty \leq 1$, then $\|u(t)\|_\infty \leq 1$ for $0\leq t\leq T$.

For the temporal accuracy test, the spatial step is fixed at $h = 1/512$. The $L^\infty$ numerical errors at $T = 10$ are evaluated. The reference solution is obtained using the $\mathcal{S}^{[4],4}$ \eqref{eq:richardson schemes} splitting operator with $h = 1/512$ and $\tau = 1/2000$. As shown in \eqref{tab:convergence_ac}, the $\mathcal{S}^{[2],1}$ splitting operator achieves a second-order temporal accuracy, while the $\mathcal{S}^{[4],4}$ and $\mathcal{S}^{[4]}_{\rm{com}}$ splitting operators exhibit a fourth-order temporal accuracy as the time step size refines. In contrast, we observe that $\mathcal{S}^{[4]}_{\rm{neg}}$ becomes unstable at step sizes $ \tau = 1/40$  and  $1/80$, leading to significant error amplification. As the time step is further refined, the accuracy approaches the fourth order.

The evolution results of solutions using these three splitting methods are shown in \eqref{fig:ac_solution}. These results demonstrate that the evolution of the solution using the $\mathcal{S}^{[4],4}$ splitting operator is the same as that of the stabilized Strang splitting operator, preserving the maximum principle. However, the $\mathcal{S}^{[4]}_{\rm neg}$ splitting operator evolves incorrectly at $T = 10$. In \eqref{fig:ac_properties}, the maximal value of $\mathcal{S}^{[4]}_{\rm neg}$ contradicts the maximum principle at $t = 1$. Furthermore, the energy evolution of each splitting operator shows that the evolution of $\mathcal{S}^{[4],4}$ and  $\mathcal{S}^{[4]}_{\rm com}$ coincides with that of $\mathcal{S}^{[2],1}$, while the $\mathcal{S}^{[4]}_{\rm neg}$ splitting operator begins to increase around $t = 1$. Both the energy dissipation law and the maximum principle of the AC equation are violated in the $\mathcal{S}^{[4]}_{\rm neg}$ splitting operator with $\tau = 1/40$. This illustrates the instability of the high-order splitting methods with negative step coefficients for coarse time steps. The reason for this is that the $\mathcal{E}_A$ operator becomes unstable at negative time steps, necessitating stabilization of the $\mathcal{E}_A$ operator by applying certain CFL condition restrictions on step size. In contrast, the properties of MPE splitting methods  and splitting methods with complex time step coefficients are excellent.

\begin{table}[htb!]
\renewcommand\arraystretch{1.5}
\begin{center}
\def\temptablewidth{1\textwidth}
\caption{$L^\infty$-errors of numerical solutions of four splitting operators at time $T = 10$ to the AC equation \eqref{eq:AC} for different splitting steps with $512\times512$ Fourier modes.}\label{tab:convergence_ac}
\vspace{-0.2in}
{\rule{\temptablewidth}{1.1pt}}
\begin{tabular*}{\temptablewidth}{@{\extracolsep{\fill}}ccccc}
   $\tau$   &$\frac1{40}$     &$\frac1{80}$     &$\frac1{160}$
   &$\frac1{320}$\\  \hline
   
  Error of $\mathcal{S}^{[2],1}$    & $1.1231\times 10^{-4}$   & $2.8517\times 10^{-5}$   & $7.1668\times 10^{-6}$ & $1.7945\times 10^{-6}$ \\[3pt]
Rate  of $\mathcal{S}^{[2],1}$   & -- & $1.9775$ &$1.9924$  &$1.9978$ \\[3pt]\hline

  Error of $\mathcal{S}^{[4],4}$    & $5.9227\times 10^{-7}$   & $5.0408\times 10^{-8}$   & $3.7091\times 10^{-9}$ & $ 2.5807\times 10^{-10}$ \\[3pt]
Rate  of $\mathcal{S}^{[4],4}$   & -- & $3.5544$ &$3.7645$  &$3.8452$ \\[3pt]\hline

  Error of $\mathcal{S}^{[4]}_{\rm{com}}$    & $2.9496\times 10^{-7}$   & $2.4213\times 10^{-8}$   & $1.7491\times 10^{-9}$ & $ 1.2548\times 10^{-10}$  \\[3pt]
Rate  of $\mathcal{S}^{[4]}_{\rm{com}}$   & -- &$3.6066$ &$3.7911$  &$3.8010$ \\[3pt]\hline

  Error of $\mathcal{S}^{[4]}_{\rm{neg}}$    & $1.5280$   & $2.1804\times 10^{-7}$   & $4.8148\times 10^{-9}$ & $3.3334\times 10^{-10}$   \\[3pt]
Rate  of $\mathcal{S}^{[4]}_{\rm{neg}}$   & -- &  $22.7406$ &$5.5010$  &$3.8524$
\end{tabular*}
{\rule{\temptablewidth}{1pt}}
\end{center}
\end{table}

\begin{figure}[htbp]
    \centering
    \includegraphics[trim = {0in 0in 0in 0in},clip,width = 0.9\textwidth]{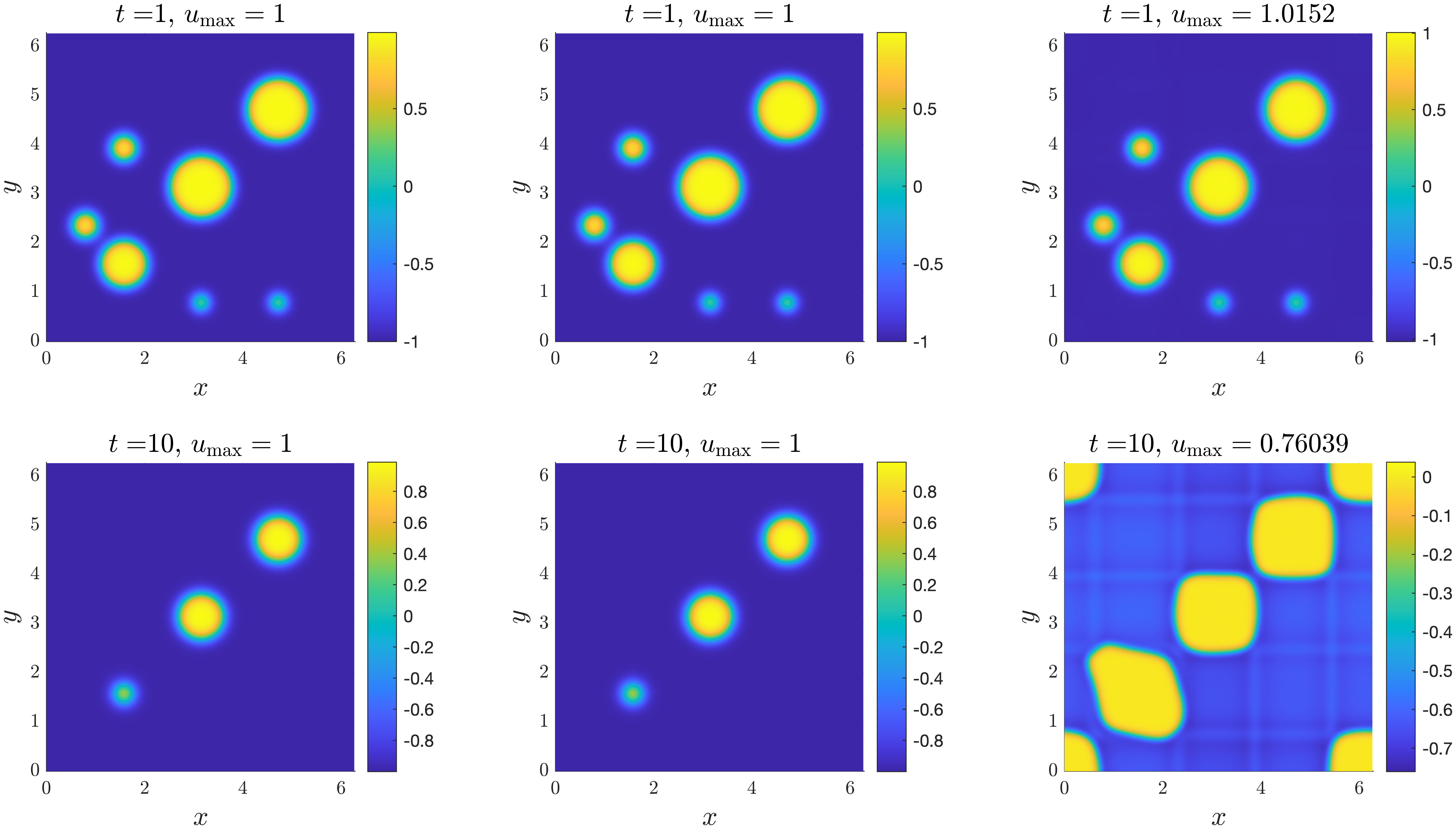} 
    \caption{Numerical solutions of the AC equation \eqref{eq:AC} at time $t=1$ (top row) and $t=10$ (bottom row), computed with a fixed time step $\tau = 1/40$. Solutions in each row, from left to right, correspond to the splitting operators: $\mathcal{S}^{[2],1}$, $\mathcal{S}^{[4],4}$, $\mathcal{S}^{[4]}_{\rm com}$, and  $\mathcal{S}^{[4]}_{\rm neg}$.}
    \label{fig:ac_solution}
\end{figure}

\begin{figure}[htbp]
    \centering
    \includegraphics[trim = {0in 0in 0in 0in},clip,width = 0.5\textwidth]{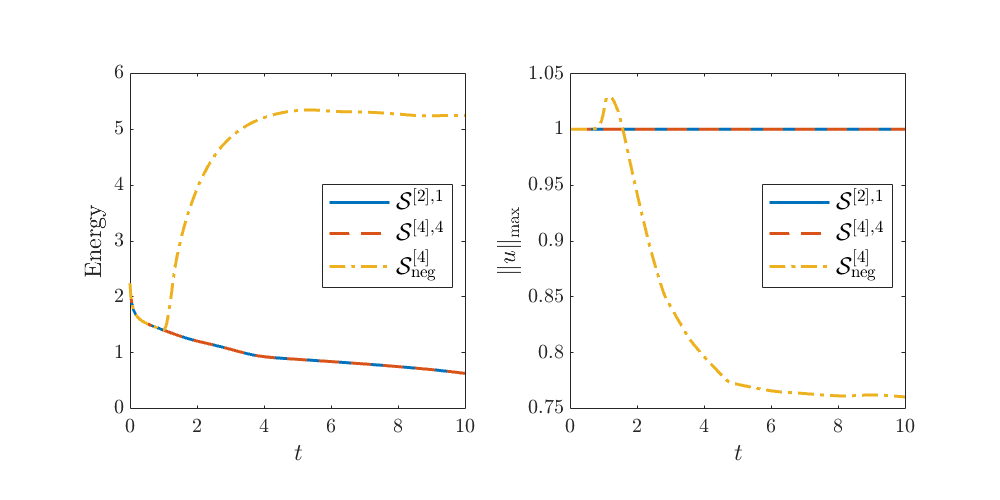} 
    \caption{The temporal evolution of energy (left) and evolution of maximum norm (right) of the
four operator splitting methods with $\tau = 1/40$ until $T=10$.}
    \label{fig:ac_properties}
\end{figure}

\section{Conclusions}\label{section:conclusions}
In this work, we established for the first time a framework for analyzing the stability of arbitrarily high-order MPE splitting methods for nonlinear semilinear parabolic equations, under mild regularity assumptions on the reaction term. Based on this stability result, it is then proved rigorously that MPE splitting methods achieve the desired order of accuracy in $L^p$ norm.  Numerical experiments corroborated the theoretical findings, showcasing the accuracy and stability. In future work, we aim to investigate the structure preservation of MPE splittings (such as the maximum principle, energy dissipation for gradient flows) for other evolution equations.



\section*{Acknowledgments}
The work of X. Duan is supported by National Natural Science Foundation of China (No. 12101424). C. Quan is supported by the National Natural Science Foundation of China (Grant No. 12271241), Guangdong Provincial Key Laboratory of Mathematical Foundations for Artificial Intelligence (2023B1212010001), Guangdong Basic and Applied Basic Research Foundation (Grant No. 2023B1515020030),  Shenzhen Science and Technology Innovation Program (Grant No. JCYJ20230807092402004), and Hetao Shenzhen-Hong Kong Science and Technology  Innovation Cooperation Zone Project (No. HZQSWS-KCCYB-2024016). The work of J. Yang is supported by the National Science Foundation of China (No.12271240, 
12426312), the fund of the Guangdong Provincial Key Laboratory of Computational Science and Material Design, China (No.2019B030301001), and the Shenzhen Natural Science Fund (RCJC20210609103819018).

\appendix

\section{Lie derivative}\label{appendix:lie derivative}
We employ the calculus of Lie derivative (see, e.g., \cite[Sect.III.5]{hairergeometric}, \cite[Sect.IX.1.4]{hundsdorfer2013numerical}, or \cite[Appendix.A]{koch2013error}). Consider an initial value problem of the form
\begin{equation}\label{eq:Lie_initial_problem}
    \begin{aligned}
        u^{\prime}(t)=F(u(t)), \quad 0 \leq t \leq T, \quad u(0)=u^0,
    \end{aligned}
\end{equation}
where the (unbounded nonlinear) operator $F: D(F) \subset X \rightarrow X$ is defined on a non-empty subset of the Banach space $\left(X,\|\cdot\|_{X}\right)$ and $u^0$ is the initial data. Formally, the exact solution to the evolutionary problem \eqref{eq:Lie_initial_problem} is given by
\begin{equation}
    \begin{aligned}
    u(t)=\mathcal{E}_{F}(t) u^0, \quad 0 \leq t \leq T,
    \end{aligned}
\end{equation}
with the evolution operator $\mathcal{E}_{F}$ depending on the actual time and acting on the initial value. The Lie derivative $D_{F}$ associated with $F$ is defined by the relation
\begin{equation}
    \begin{aligned}
D_{F} G (u^0)& =\left.\frac{\mathrm{d}}{\mathrm{~d} t}\right|_{t=0} G\left(\mathcal{E}_{F}(t) u^0\right)=G^{\prime}\left(u^0\right) F\left(u^0\right),
\end{aligned}
\end{equation}
for any (unbounded nonlinear) operator $G: D(G) \subset X \rightarrow X$, defined on a suitable domain, with $G'$ denoting the  Fr\'{e}chet-derivative of $G$, and the operator $\left(\mathrm{e}^{t D_{F}}\right)_{0 \leq t \leq T}$  is defined  by
\begin{equation}
    \begin{aligned}
\mathrm{e}^{t D_{F}} G(u^0)=G(\mathcal{E}_{F}(t) u^0), \quad 0 \leq t \leq T.
    \end{aligned}
\end{equation}
In particular, if $G= \mathcal{I}$ is the identity operator, we obtain $\mathrm{e}^{t D_{F}}\mathcal{I} (u^0)=\mathcal{E}_{F}(t)u^0$ and $D_{F}\mathcal{I} (u^0)=F(u^0)$. For brevity, we will omit the notation of $\mathcal{I}$ in subsequent expressions.

It is evident that the evolution operator forms a local one-parameter group
\begin{equation}
    \begin{aligned}
\mathrm{e}^{(t_1+t_2) D_{F}}=\mathrm{e}^{t_1 D_{F}} \mathrm{e}^{t_2 D_{F}}=\mathrm{e}^{t_2 D_{F}} \mathrm{e}^{t_1 D_{F}}, \quad 0 \leq t_1+t_2 \leq T,\left.\quad \mathrm{e}^{t_1 D_{F}}\right|_{t_1=0}=\mathcal{I}, 
    \end{aligned}
\end{equation}
since $\mathcal{E}_F(t_1+t_2) u^0 = \mathcal{E}_F(t_2)(\mathcal{E}_F(t_1) u^0)$ by the local existence and uniqueness of the solution and consequently $\mathrm{e}^{(t_1+t_2) D_F} G(u^0) = G\left(\mathcal{E}_F(t_1+t_2) u^0\right) = G\left(\mathcal{E}_F(t_2) (\mathcal{E}_F(t_1) u^0)\right) = \mathrm{e}^{t_1 D_F} \mathrm{e}^{t_2 D_F} G (u^0)$. It is also important to note that the composition of evolution operators acts in reverse order, i.e., the following holds
\begin{equation}
    \begin{aligned}
\mathrm{e}^{t_1 D_{F_{1}}} \mathrm{e}^{t_2 D_{F_{2}}} G (u^0)=G\left(\mathcal{E}_{F_2}(t_2)( \mathcal{E}_{F_1}(t) u^0)\right), \quad 0 \leq t_1+t_2 \leq T.
    \end{aligned}
\end{equation}
We omit the parentheses later for short. That is $G\left(\mathcal{E}_{F_2}(t_2)( \mathcal{E}_{F_1}(t_1) u^0)\right)=G\left(\mathcal{E}_{F_2}(t_2)\mathcal{E}_{F_1}(t_1) u^0\right)$.

Additionally, we have the rule
\begin{equation}
    \begin{aligned}
\frac{\mathrm{d}}{\mathrm{d} t} \mathrm{e}^{t D_{F}}G(u^0)=\mathrm{e}^{t D_{F}} D_{F}G(u^0)=D_{F} \mathrm{e}^{t D_{F}}G(u^0), \quad 0 \leq t \leq T.
    \end{aligned}
\end{equation}
We also note that the identity
\begin{equation}
    \begin{aligned}
    \mathrm{e}^{t D_{F}}=\mathcal{I}+\left.\mathrm{e}^{\sigma D_{F}}\right|_{\sigma=0} ^{t}=\mathcal{I}+\int_{0}^{t} \frac{\mathrm{~d}}{\mathrm{~d} \sigma} \mathrm{e}^{\sigma D_{F}} \mathrm{~d} \sigma=\mathcal{I}+\int_{0}^{t} \mathrm{e}^{\sigma D_{F}} D_{F} \mathrm{~d} \sigma, \quad 0 \leq t \leq T, 
    \end{aligned}
\end{equation}
which, as justified by the above considerations, implies the (Taylor-like) expansion
\begin{equation}\label{eq:e^tDF}
    \begin{aligned}
    \mathrm{e}^{t D_{F}}=\sum_{j=0}^{s-1} \frac{1}{j!} t^{j} D_{F}^{j}+\int_{T_{s}} \mathrm{e}^{\sigma_{s} D_{F}} D_{F}^{s} \mathrm{~d} \sigma, \quad 0 \leq t \leq T, \quad s \in \mathbb{N}^+,
    \end{aligned}
\end{equation}
where we denote $T_{s}=\left\{\sigma=\left(\sigma_{1}, \sigma_{2}, \ldots, \sigma_{s}\right) \in \mathbb{R}^{s}: 0 \leq \sigma_{s} \leq \ldots \leq \sigma_{1} \leq \sigma_{0}=t\right\}$ and, as common usage, set $D_{F}^{0}=\mathcal{I}$. For integers $j \geq 0$, the linear operators $\varphi_j(tD_F)$ are defined through
\begin{equation}\label{eq:phi}
    \begin{aligned}
        \varphi_0(tD_F) = \mathrm{e}^{tD_F},\quad\varphi_j(tD_F) = \frac{1}{(j-1)!}\int_0^1 \sigma^{j-1}\mathrm{e}^{(1-\sigma)tD_F}\mathrm{~d}\sigma,
    \end{aligned}
\end{equation}
with the recurrence relation
\begin{equation}\label{eq:phi_recurrence}
    \begin{aligned}
        \varphi_{j}\left(t D_{F}\right)=\frac{1}{j!} \mathcal{I}+\varphi_{j+1}\left(t D_{F}\right) t D_{F}, \quad j \geq 0.
    \end{aligned}
\end{equation}

The Lie-commutator of two nonlinear operators $F,G$ is given by $\operatorname{ad}_{F}(G)(v) = [F,G](v) = F'(v)G(v) - G'(v)F(v)$; in particular, when $F$ and $G$ are linear, the above relation reduces to $\operatorname{ad}_{F}(G) =FG-GF$. In accordance with the above definition, we further set $\operatorname{ad}_{D_{F}}\left(D_{G}\right)(v) = \left[D_{F}, D_{G}\right] (v)=D_{F} D_{G} (v)-D_{G} D_{F} (v)=G^{\prime}(v) F(v) - F^{\prime}(v) G(v)$. By definition, the commutator $[D_F , D_G] = -D_{[F,G]}.$
More generally, the iterated Lie-commutators are defined by induction
\begin{equation}\label{eq:ad}
    \begin{aligned}
\operatorname{ad}_{D_{F}}^{0}\left(D_{G}\right)=D_{G}, \quad \operatorname{ad}_{D_{F}}^{j}\left(D_{G}\right)=\left[D_{F}, \operatorname{ad}_{D_{F}}^{j-1}\left(D_{G}\right)\right], \quad j \geq 1, 
    \end{aligned}
\end{equation}
and 
\begin{equation}\label{eq:D_ad}
    \begin{aligned}
\operatorname{ad}_{D_{F}}^{j}\left(D_{G}\right)= (-1)^j D_{\operatorname{ad}_{{F}}^{j}\left({G}\right)}, \quad j \geq 1.
    \end{aligned}
\end{equation}

\section{Proof of \eqref{thm:exact_sol_bound}}\label{appendix:proof of thm}
Before the proof, we introduce the following lemma as a preparatory step.

 \begin{lemma}\label{lemma:D_alpha_exactu_uniform_bound}Suppose that the initial data $u^0 \in H^m_\text{per}(\Omega)$ for $m\geq 1$, the spatial dimension $d=1,2,3$, $f \in C^{m}(\mathbb{R})$, and \eqref{assumption:f'_bounded} is satisfied. Let $u(t)$ be the exact solution to \eqref{eq:semilinear equation} with periodic boundary conditions corresponding to the initial data $u^0$. Then, for any $T\geq 0$, it holds that 
\begin{equation}\label{eq:uni_bound}
    \quad\left\|\partial_x^{\alpha} u(t)\right\|_{2} \leq  \mathrm{e}^{\kappa t}\left(\left\|\partial_x^{\alpha} u^{0}\right\|_{2}+C\sqrt{t}\right),\quad 0\leq t\leq T,
\end{equation}
 where $C\geq 0 $ depends on $(\|u^0\|_{H^{m-1}(\Omega)},m,d,T,\kappa)$ and $\alpha$ is a multi-index with $|\alpha| = m$.
\end{lemma}
\begin{proof}
    We first prove the case when $|\alpha| = 0$. It is already known that
    \begin{equation}
        |f(u)| \leq |f(u) - f(0)|+|f(0)| \leq \kappa|u|+|f(0)|.
    \end{equation}
    Hence, we can see straight away that
\begin{equation}
\begin{aligned}
    2\|u\|_{2}\partial_{t}\|u\|_{2}& =\partial_{t}\|u\|_{2}^{2} =\partial_ t \int_{\Omega} u^{2} \dx=2\int_{\Omega} u_ t u \dx \\
    & =2 \int_{\Omega}\left(\Delta u+f(u)\right) u \dx 
    =-2\|\nabla u\|_{2}^{2}+ 2\int_{\Omega} uf(u) \dx \\
    &\leq 2\int_{\Omega}| uf(u)| \dx \leq  2\int_{\Omega}\big(\kappa|u^2|+|f(0)u|\big)\dx \leq 2\kappa\|u\|_{2}^2 + 2|f(0)|\|u\|_{1}\\
    & \leq 2\kappa\|u\|_{2}^2 +2\left|\Omega\right|^{\frac{1}{2}}|f(0)|\|u\|_{2},
\end{aligned}
\end{equation}
where the last inequality is obtained by the H{\"o}lder inequality.
Without loss of generality, we have
\begin{equation}
    \partial_t\|u\|_{2}\leq \kappa\|u\|_{2} + \left|\Omega\right|^{\frac{1}{2}}|f(0)|.
\end{equation}
By Gr{\"o}nwall's lemma, we obtain
$
    \|u(t)\|_{2} \leq \mathrm{e}^{\kappa t}(\left\|u^0\right\|_{2}+C_0 t),
$
where $C_0=\left|\Omega\right|^{\frac{1}{2}}|f(0)|$.

Next, we consider the case when $|\alpha| = 1$. We have
\begin{equation}
    \begin{aligned}
         \partial_t\left\|\partial_x^{\alpha} u\right\|_{2}^{2} & =2 \int_{\Omega} \partial_x^{\alpha} u \partial_x^{\alpha}\left(u_{t}\right) \dx=2 \int_{\Omega} \partial_x^{\alpha} u\left(\Delta \partial_x^{\alpha} u+\partial_x^{\alpha}f(u) \right)\dx \\
&=-2\left\|\nabla \partial_x^{\alpha} u\right\|_{2}^{2}+2\int_{\Omega} \partial_x^{\alpha} u\partial_x^{\alpha}f(u) \dx \leq 2\int_{\Omega} \partial_x^{\alpha} uf'(u) \partial_x^{\alpha} u \dx\\
&=2\int_{\Omega} |\partial_x^{\alpha} u|^2|f'(u)|  \dx \leq  2 \kappa \left\|\partial_x^{\alpha} u\right\|_{2}^{2}.
    \end{aligned}
\end{equation}
Again, by Gr{\"o}nwall's lemma, we obtain
\begin{equation}\label{eq:uni_bound_alpha=1}
    \left\|\partial_x^\alpha u(t)\right\|_{2}\leq \mathrm{e}^{\kappa t}\left\|\partial_x^\alpha u^0\right\|_{2}.
\end{equation}

For $|\alpha| \geq 2$, we suppose $\partial_x^{\alpha}u=\partial^{\alpha_1}_{x}\partial_x^{\widetilde{\alpha}}u$, where $|\alpha_1|=1$ and $\widetilde{\alpha} = \alpha-\alpha_1$. Then
\begin{equation}\label{eq:uni_bound_process_alpha=k}
    \begin{aligned}
         \partial_{t}\left\|\partial_x^{\alpha} u\right\|_{2}^2 =&-2\left\|\nabla \partial_x^{\alpha} u\right\|_{2}^{2}+2\int_{\Omega} \partial_x^{\alpha} u\partial_x^{\alpha}f(u) \dx\\
         =& -2\left\|\nabla \partial_x^{\alpha} u\right\|_{2}^{2} -2\int_{\Omega} \partial_x^{\widetilde{\alpha}}f(u) \partial^{\alpha_1}_{x}\partial_x^\alpha u \dx
         \\
         \leq & -2\left\|\nabla \partial_x^{\alpha} u\right\|_{2}^{2} +2\left\|\partial^{\alpha_1}_{x}\partial_x^\alpha u\right\|_{2}^{2}+\frac{1}{2}\left\|\partial_x^{\widetilde{\alpha}}f(u)\right\|_{2}^{2}\leq \frac{1}{2}\left\|\partial_x^{\widetilde{\alpha}}f(u)\right\|_{2}^{2}.
    \end{aligned}
\end{equation}
We only need to estimate $\|\partial_x^{\widetilde{\alpha}}f(u)\|_{2}$. We have
\begin{equation}
    \begin{aligned}
\partial_x^{\widetilde{\alpha}}f(u)= |\widetilde{\alpha}|f'(u)\partial_x^{\widetilde{\alpha}}u + \sum_{2\leq l\leq|\widetilde{\alpha}|}f^{(l)}(u)\sum_{\substack{\mu_{1}+\mu_{2}+\cdots+\mu_{l}=\widetilde{\alpha} \\
 1\leq|\mu_{1}|,| \mu_{2}|,\cdots,|\mu_{l}|\leq|\widetilde{\alpha}|-1}}   \partial_x^{\mu_1}u\partial_x^{\mu_2}u\cdots \partial_x^{\mu_l}u.
    \end{aligned}
\end{equation}

If $|\alpha|=2$, then $|\widetilde{\alpha}|=1$. Hence, by \eqref{eq:uni_bound_alpha=1}, we have
\begin{equation}
\left\|\partial_x^{\widetilde{\alpha}}f(u)\right\|_{2}=\left\|f'(u)\partial_x^{\widetilde{\alpha}}u\right\|_{2}\leq \kappa \left\|\partial_x^{\widetilde{\alpha}}u\right\|_{2}\leq \kappa \mathrm{e}^{\kappa T}\left\|u^0\right\|_{H^1}.
\end{equation}
Therefore, from \eqref{eq:uni_bound_process_alpha=k}, we have
\begin{equation}
    \partial_{t}\left\|\partial_x^{\alpha} u\right\|_{2}^2\leq \frac{\kappa^2\mathrm{e}^{2\kappa T}}{2}\|u^0\|_{H^1}^2,
\end{equation}
which implies
\begin{equation}
    \left\|\partial_x^\alpha u(t)\right\|_{2}\leq\sqrt{\left\|\partial_x^\alpha u^0\right\|_{2}^2+\frac{\kappa^2\mathrm{e}^{2\kappa T}}{2}\left\|u^0\right\|_{H^1}^2t}\leq \left\|\partial_x^\alpha u^0\right\|_{2} + C\sqrt{t},
\end{equation}
where $C=\kappa \mathrm{e}^{\kappa T}\|u^0\|_{H^1}$. So \eqref{eq:uni_bound} is verified for $|\alpha|\leq 2$. 

Now, we suppose \eqref{eq:uni_bound} holds for $|\alpha|\leq m-1$ with $m\geq 3$. Then we have
\begin{equation}
    \underset{0\leq t\leq T}{\sup}\left\|u(t)\right\|_{H^{m-1}}\leq  M_{m-1},
\end{equation}
where $M_{m-1}$ is a constant that depends on $(\|u^0\|_{H^{m-1}},m-1,d,T,\kappa)$.

When $|\alpha|=m$, we have
\begin{equation}
    \begin{aligned}
    \left\|\partial_x^{\widetilde{\alpha}}f(u)\right\|_{2} \leq |\widetilde{\alpha}|\left\|f'(u)\partial_x^{\widetilde{\alpha}}u\right\|_{2} + \sum_{l}\sum_{r}  \left\|f^{(l)}(u)\right\|_{\infty}\left\|\partial_x^{\mu_1}u\partial_x^{\mu_2}u\cdots \partial_x^{\mu_l}u\right\|_{2},
    \end{aligned}
\end{equation}
where $|\mu_i|\leq m-2$. Since $m\geq 3$ and $|\mu|=m-1$, there are at most 2 derivatives $\partial_x^{\mu_i}u$ with order $|\mu_i|=m-2$. Without loss of generality, we suppose $\partial_x^{\mu_1}u$ and $\partial_x^{\mu_2}u$ are the top two derivatives in the order. Then by Sobolev embedding theorem, $H_\text{per}^{m-1}(\Omega)\hookrightarrow W_\text{per}^{m-3,\infty}(\Omega)$, $H_\text{per}^{m-1}(\Omega)\hookrightarrow W_\text{per}^{m-2,4}(\Omega)$, so we have
\begin{equation}\label{eq:uni_bound_alpha=k_estimate_Du}
\begin{aligned}
\left\|f'(u)\partial_x^{\widetilde{\alpha}}u\right\|_{2}&\leq\kappa\left\|\partial_x^{\widetilde{\alpha}}u\right\|_{2}\leq \kappa \|u\|_{H^{m-1}}\leq \kappa M_{m-1},\\
    \left\|\partial_x^{\mu_1}u\partial_x^{\mu_2}u\cdots \partial_x^{\mu_l}u\right\|_{2}&\leq \left\|\partial_x^{\mu_1}u\right\|_{4}\left\|\partial_x^{\mu_2}u\right\|_{4}\left\|\partial_x^{\mu_3}u\right\|_{\infty}\cdots\left\|\partial_x^{\mu_l}u\right\|_{\infty}\leq \bar{C}\|u\|^l_{H^{m-1}}\leq \bar{C} M_{m-1}^l,
\end{aligned}
\end{equation}
where $\bar{C}$ is a constant.

Moreover, since $\|u\|_{\infty}$ is bounded by $\|u\|_{H^{2}}$ and $f^{(l)}(u)$ is bounded on bounded domain, so we have
\begin{equation}\label{eq:uni_bound_alpha=k_estimate_f^l}
    \left\|f^{(l)}(u)\right\|_{\infty}\leq \bar{C}_l,
\end{equation}
where $\bar{C}_l$ is a constant that depends on $\|u^0\|_{H^{2}}$. 

Now, let us return to \eqref{eq:uni_bound_process_alpha=k}, by \eqref{eq:uni_bound_alpha=k_estimate_Du} and \eqref{eq:uni_bound_alpha=k_estimate_f^l},  
we know that there exists a constant $C_m$ that depends on $(\|u^0\|_{H^{m-1}},k,d,T,\kappa)$ such that
\begin{equation}
    \partial_{t}\left\|\partial_x^{\alpha} u\right\|_{2}^2\leq \frac{1}{2}\left\|\partial_x^{\widetilde{\alpha}}f(u)\right\|_{2}^{2}\leq C_m^2,
\end{equation}
which implies
\begin{equation}
    \left\|\partial_x^{\alpha} u(t)\right\|_{2}\leq \sqrt{\left\|\partial_x^{\alpha} u^0\right\|_{2}^2+C_m^2 t}\leq\left\|\partial_x^{\alpha} u^0\right\|_{2} + C_m\sqrt{t}.
\end{equation}
By induction on $|\alpha|$, we complete the proof.
\end{proof}

Then the proof of \eqref{thm:exact_sol_bound} can be finished as follows. 

\begin{proof}
The results of $H^m$-norm bounds are trivial from Lemma \ref{lemma:D_alpha_exactu_uniform_bound} by summation. We only need to prove the $L^p$-norm bounds of the exact solution.
Firstly, we consider the case of $p\in[1,\infty)$. Similarly, we obtain
\begin{equation}
\begin{aligned}
    p\|u\|^{p-1}_{p}\partial_{t}\|u\|_{p}& =\partial_{t}\|u\|_{p}^{p} =\partial_ t \int_{\Omega} |u|^{p} \dx=p\int_{\Omega} u|u|^{p-2}u_ t  \dx \\
    & =p \int_{\Omega}u|u|^{p-2}\left(\Delta u+f(u)\right)  \dx 
    = p \int_{\Omega}-\nabla\left(u|u|^{p-2}\right)\nabla u  \dx + u|u|^{p-2}f(u)  \dx\\
    &=-p(p-1)\int_{\Omega} |u|^{p-2}|\nabla u|^2 +  p\int_{\Omega}u|u|^{p-2}f(u) \dx\\
    &\leq p\int_{\Omega}u|u|^{p-2}f(u) \dx \leq  p\int_{\Omega}|u|^{p-1}\big(\kappa|u|+|f(0)|\big)\dx \\
    & \leq p\kappa\|u\|_{p}^p +p|f(0)|\left|\Omega\right|^{\frac{1}{p}}\|u\|^{p-1}_{p},
\end{aligned}
\end{equation}
where the last inequality is obtained by the H{\"o}lder inequality. This yields that 
\begin{equation}
    \begin{aligned}
        \partial_{t}\|u\|_{p}\leq \kappa\|u\|_{p} +|f(0)|\left|\Omega\right|^{\frac{1}{p}}. 
    \end{aligned}
\end{equation}
By Gr{\"o}nwall's lemma, we obtain
\begin{equation}
    \|u(t)\|_{p} \leq \mathrm{e}^{\kappa t}\left(\left\|u^0\right\|_{p}+\left|\Omega\right|^{\frac{1}{p}}|f(0)| t\right) \leq \mathrm{e}^{\kappa t}\left(\left\|u^0\right\|_{p}+C\right),
\end{equation}
where $C=\left|\Omega\right|^{\frac{1}{p}}|f(0)|T$.

 Regarding the case of $p=\infty$, we use the Duhamel's principle
 \begin{equation}
     \begin{aligned}
         u(t) = \mathrm{e}^{t\Delta}u^0 + \int_{0}^t \mathrm{e}^{(t-s)\Delta}f(u(s)) \mathrm{~d} s.
     \end{aligned}
 \end{equation}
Hence, we can obtain that
 \begin{equation}
     \begin{aligned}
         \|u(t)\|_\infty &\leq \|\mathrm{e}^{t\Delta}u^0\|_\infty + \int_{0}^t \|\mathrm{e}^{(t-s)\Delta}f(u(s))\|_\infty \mathrm{~d} s\\
         & \leq \|u^0\|_\infty + \int_{0}^t \|f(u(s))\|_\infty \mathrm{~d} s \leq  \|u^0\|_\infty + |f(0)|t +\int_{0}^t\kappa\| u(s)\|_\infty \mathrm{~d} s,
     \end{aligned}
 \end{equation}
yielding
\begin{equation}
    \|u(t)\|_\infty \leq \mathrm{e}^{\kappa t}\left(\left\|u^0\right\|_{\infty}+|f(0)|t\right),
\end{equation}
where the stability of $\mathcal{E}_A(t)$ in Lemma \ref{lemma:E_A_stability} and the Gr{\"o}nwall's lemma are employed.
\end{proof}

\bibliographystyle{siamplain}
\bibliography{references}
\end{document}